\documentclass[a4paper,11pt,leqno]{article}
\usepackage{amsmath,amssymb}
\usepackage{amsfonts}
\usepackage{eucal}
\usepackage{amsthm}
\numberwithin{equation}{section}

\newcommand{\bigpare}[1]{\bigl(#1\bigr)}
\newcommand{\biggpare}[1]{\biggl(#1\biggr)}

\newcommand{\bigbrac}[1]{\bigl[#1\bigr]}

\newcommand{\bigset}[2]{\bigl\{#1\bigm|#2\bigr\}}

\newcommand{\norm}[1]{\| #1 \|}
\newcommand{\bignorm}[1]{\bigl\| #1 \bigr\|}

\newcommand{\abs}[1]{| #1 |}
\newcommand{\bigabs}[1]{\bigl| #1 \bigr|}

\newcommand{\biggabs}[1]{\biggl| #1 \biggr|}

\newcommand{\jap}[1]{\langle #1 \rangle}

\def\a{\alpha}
\def\b{\beta}

\def\d{\delta}
\def\e{\varepsilon}
\def\f{\varphi}
\def\g{\psi}

\def\l{\lambda}
\def\m{\mu}
\def\n{\nu}
\def\o{\omega}
\def\s{\sigma}
\def\t{\tau}
\def\x{\xi}
\def\y{\eta}
\def\z{\zeta}
\def\th{\theta}
\newcommand{\C}{\Gamma}
\newcommand{\F}{\Phi}
\newcommand{\G}{\Psi}
\renewcommand{\L}{\Lambda}
\renewcommand{\O}{\Omega}

\def\re{\mathbb{R}}

\def\ze{\mathbb{Z}}

\def\pa{\partial}

\newcommand{\supp}{\text{{\rm supp}\;}}
\newcommand{\trace}{\text{{\rm Tr}} }

\DeclareMathOperator*{\slim}{s-lim}

\newcommand{\Ran}{\text{\rm Ran\;}}
\newcommand{\WF}{\mbox{WF}}
\newcommand{\1}{\mbox{I}}
\newcommand{\2}{\mbox{II}}
\newcommand{\3}{\mbox{III}}

\newtheorem{thm}{Theorem}[section]
\newtheorem{lem}[thm]{Lemma}

\newtheorem{cor}[thm]{Corollary}

\theoremstyle{definition}
\newtheorem{defn}{Definition}[section]
\newtheorem{ass}{Assumption}

\theoremstyle{remark}
\newtheorem{rem}{Remark}[section]

\numberwithin{equation}{section}

\title{Remarks on the Fundamental Solution to Schr\"odinger Equation with 
Variable Coefficients}
\author{
Kenichi I{\sc to}%
\footnote{Graduate School of Pure and Applied Sciences, University of Tsukuba,
1-1-1 Tennodai, Tsukuba Ibaraki, 305-8571 Japan. 
E-mail: \texttt{ito\_ken@math.tsukuba.ac.jp}. }
\ and Shu N{\sc akamura}%
\footnote{Graduate School of Mathematical Sciences, 
University of Tokyo, 3-8-1 Komaba, Meguro Tokyo, 
153-8914 Japan. 
E-mail: {\tt shu@ms.u-tokyo.ac.jp}.  
Partially supported by JSPS Grant Kiban (B) 17340033 (2005--2008); Kiban (A) 
21244008 (2009-2013).} }
\begin{document}
\maketitle

\begin{abstract}
We consider Schr\"odinger operators $H$ on $\re^n$ with variable coefficients. 
Let $H_0=-\frac12\triangle$ be the free Schr\"odinger operator and we suppose 
$H$ is a ``short-range'' perturbation of $H_0$. 
Then, under the nontrapping condition, we show the time evolution operator: 
$e^{-itH}$ can be written as a product of the free evolution operator $e^{-itH_0}$ 
and a Fourier integral operator $W(t)$, which is associated to the canonical relation 
given by the classical mechanical scattering. 
We also prove a similar result for the wave operators. 
These results are analogous to results by Hassell and Wunsch \cite{HW1,HW2}, 
but the assumptions, the proof and the formulation of results are considerably different.
The proof employs an Egorov-type theorem similar to those used in previous works 
by the authors \cite{Na1,Na2,IN1} combined with a Beals-type characterization of 
Fourier integral operators. 
\end{abstract}

%%%%%%%%%%%%%%%%%%%%%%%%%%%%%%%%%%%%%%%%%
%%%%%%%%%%%%  Section 1  %%%%%%%%%%%%%%%%
%%%%%%%%%%%%%%%%%%%%%%%%%%%%%%%%%%%%%%%%%

\section{Introduction}
We consider Schr\"odinger equations on $\re^n$ with $n\geq 1$ 
of the following form:
\begin{align*}
& i\frac{\pa}{\pa t}\g(t,x) =H \g(t,x), \quad t\in\re,x\in\re^n,\\
&\g(0,x)=\g_0(x)\in L^2(\re^n),\\
&H=-\frac12\sum_{j,k=1}^n \frac{\pa}{\pa x_j}a_{jk}(x)\frac{\pa}{\pa x_k} +V(x), 
\end{align*}
where $a_{jk}(x)$ and $V(x)$ are real-valued $C^\infty$-functions on $\re^n$. 

\begin{ass}\label{ass0aa}
There exists $\m>0$ such that for any $\a\in\ze_+^n$ 
\[
\bigabs{\pa_x^\a\bigpare{a_{jk}(x)-\d_{jk}}}\leq C_\a\jap{x}^{-\m-|\a|}, 
\quad \bigabs{\pa_x^\a V(x)}\leq C_\a\jap{x}^{2-\m-|\a|}
\]
for $x\in\re^n$ with some $C_\a>0$. 
\end{ass}

It is well-known that under our assumptions $H$ is essentially self-adjoint on 
$C_0^\infty(\re^n)$. We denote the unique self-adjoint extension by the same symbol $H$. 
Then the solution to the Schr\"odinger equation is given by $\g(t)=e^{-itH}\g_0\in L^2(\re^n)$ 
by Stone's theorem. 

We are interested in the microlocal structure of the evolution operator $e^{-itH}$. 
If $a_{jk}(x)=\d_{jk}$, i.e., if the metric is flat, then $e^{-itH}$ is represented by an oscillatory 
integral similar to Fourier integral operators (Fujiwara \cite{Fu}), though it is not a Fourier 
integral operator (FIO) in the sense of H\"ormander (\cite{Ho1, Ho2}). 
For general $H$, it is difficult to show similar representations because of the existence of 
the caustics. In this paper, we discuss different representation of the evolution operator, namely, 
we show
\[
e^{-itH}=e^{-itH_0}W(t)
\]
where $H_0=-\frac12\triangle$ is the free Schr\"odinger operator, and $W(t)$ is possibly 
an FIO. In the following, we show $W(t)$ is in fact an FIO under suitable conditions. We note 
\[
W(t)=e^{itH_0}e^{-itH},
\]
and we study the microlocal structure of $W(t)$ defined as above. 

In order to state the condition, we consider the classical mechanics associated to our 
Hamiltonian. We set
\[
p(x,\x)=k(x,\x)+V(x), \quad k(x,\x)=\frac12 \sum_{j,k=1}^n a_{jk}(x)\x_j\x_k
\]
on $T^*\re^n\cong\re^n\times \re^n$. $k(x,\x)$ is the kinetic energy, and 
$p(x,\x)$ is the classical Hamilton function. We denote the corresponding Hamilton 
vector fields by $H_p$ and $H_k$, respectively, i.e., 
\[
H_p=\sum_{j=1}^n \frac{\pa p}{\pa\x_j}\frac{\pa}{\pa x_j} 
-\sum_{j=1}^n \frac{\pa p}{\pa x_j}\frac{\pa}{\pa \x_j},
\quad 
H_k=\sum_{j=1}^n \frac{\pa k}{\pa\x_j}\frac{\pa}{\pa x_j} 
-\sum_{j=1}^n \frac{\pa k}{\pa x_j}\frac{\pa}{\pa \x_j},
\]
and we denote the Hamilton flows on $T^*\re^n$ by $\exp(tH_p)$ and $\exp(tH_k)$ 
($t\in\re$), respectively. We write
\[
T^*M\setminus 0 = \bigset{(x,\x)}{(x,\x)\in T^*M, \x\neq 0}.
\]

\begin{defn}\label{defn0101a}
Let $(x_0,\x_0)\in T^*\re^n\setminus 0$, and we denote
\[
(y(t),\y(t))=\exp(tH_k)(x_0,\x_0), \quad \text{for }t\in\re.
\]
$(x_0,\x_0)$ is said to be {\it forward (backward, resp.) nontrapping}\/ if 
\[
|y(t)|\to+\infty\quad \text{as }t\to\pm\infty.
\]
\end{defn}

If $(x_0,\x_0)$ is forward/backward nontrapping, then it is well-known that 
\[
\x_\pm =\lim_{t\to\pm\infty}\y(t)
\]
exist under Assumption~\ref{ass0aa}. Moreover, if $\m>1$, then 
\[
z_\pm =\lim_{t\to\pm\infty} (y(t)-t\y(t))
\]
are also well-known to exist (see, e.g., \cite{Na1}). These imply $y(t)\sim z_\pm +t\x_\pm$ as 
$t\to\pm\infty$. We call $(z_\pm,\x_\pm)$ the scattering data of $(x_0,\x_0)$, 
and we denote
\[
w_\pm(x_0,\x_0) =(z_\pm,\x_\pm)=\lim_{t\to\pm\infty} \exp(-tH_{p_0})
\circ\exp(tH_k)(x_0,\x_0),
\]
where $p_0(\x)=\frac12 |\x|^2$ is the free energy function. 
We note that $z_\pm(x,\x)$ and $\x_\pm(x,\x)$ are homogeneous of order 0 and 1 
with respect to $\x$, respectively, since 
both $k(x,\x)$ and $p_0(\x)$ are homogeneous of order 2 in $\x$. Moreover, 
$w_\pm$ are canonical transform on the domain where $w_\pm$ are defined. 

\begin{thm}\label{thm0101a}
Suppose Assumption~\ref{ass0aa} with $\m=2$, and suppose the global nontrapping condition, 
i.e., every $(x_0,\x_0)\in T^*\re^n\setminus 0$ is nontrapping. Then $W(t)$ is an 
FIO associated to $w_\pm$ for each $t\in\re_\pm$. 
\end{thm}

\begin{rem}\label{rem0101a}
We suppose the global nontrapping condition for the sake of simplicity. 
If we suppose $(x_0,\x_0)$ is forward nontrapping and $t>0$, then we can find 
a symbol $a(x,\x)\in S_{cl}^0(\re^n)$ such that $a_0(x_0,\x_0)\neq 0$ and 
$W(t)a(x,D_x)$ is an FIO associated to $w_+$ defined in a conic neighborhood of
$(x_0,\x_0)$. Here we have denoted the principal symbol of $a$ by $a_0$. 
The same generalization applies to the following theorems, but we do not discuss 
in detail. The proof of the above statement is same as that of the theorem. 
In fact, we need only to prove the theorem microlocally, and we prove the above 
claim to conclude the main result by using the microlocal partition of unity. 
\end{rem}

\begin{rem}\label{rem0102a}
Theorem~\ref{thm0101a} actually implies a propagation of singularity result. 
Namely, if we set 
\[
\L_\pm =\bigset{(y,\y,x,-\x)}{(y,\y)=w_\pm(x,\x)}\subset T^*\re^n\times T^*\re^n,
\]
then Theorem~\ref{thm0101a} implies $\WF(W(t))\subset \L_\pm$ for $t\in\re_\pm$, where 
$\WF(W(t))$ denotes the wave front set of the distribution kernel of $W(t)$. 
Thus, in turn, it implies
\[
\WF(W(t)u)\subset w_\pm(\WF(u))
\]
for $u\in L^2(\re^n)$ or $u\in \mathcal{E}'(\re^n)$. In fact, we have the equality 
in the above inclusion (\cite{Na1,Na2, IN1}).
\end{rem}

We now consider the general short-range case, i.e., the case when $1<\m<2$. 
Then we learn that Theorem~\ref{thm0101a} does not hold as it is, and we need to modify 
the definition of the FIOs. In \cite{Ho2}, an FIO is defined as an operator of which 
the distribution kernel is a Lagrangian distribution associated to a {\it conic}\/
Lagrangian submanifold. We need to employ a Lagrangian distribution associated to 
an {\it asymptotically conic}\/ Lagrangian manifold. Such Lagrangian submanifold 
is associated to an {\it asymptotically homogeneous}\/ canonical transform. 
We discuss these definitions in Section~4. 

We set
\[
w(t)=\exp(-tH_{p_0})\circ \exp(tH_p)\ :\ T^*\re^n\to T^*\re^n.
\]
We can show $w(t)(x,\x)=w_\pm(x,\x)+O(|\x|^{2-\m})$ as $|\x|\to\infty$ for 
$\pm t>0$, and hence $w(t)$ is asymptotically close to the homogeneous 
canonical transform $w_\pm$ if $\m>1$ (Appendix, Lemma~\ref{lem0a01a}). 

\begin{thm}\label{thm0102a}
Suppose Assumption~\ref{ass0aa} with $1<\m<2$, and suppose the global nontrapping 
condition. Then $W(t)$ is an FIO associated to $w(t)$. 
\end{thm}

Remarks~\ref{rem0101a} and \ref{rem0102a} also apply to Theorem~\ref{thm0102a}. 

\medskip

Next we consider the wave operators. Here we suppose

\begin{ass}\label{ass0ba}
There exists $\m>1$ such that for any $\a\in\ze_+^n$ 
\[
\bigabs{\pa_x^\a\bigpare{a_{jk}(x)-\d_{jk}}}\leq C_\a\jap{x}^{-\m-|\a|}, 
\quad \bigabs{\pa_x^\a V(x)}\leq C_\a\jap{x}^{-\m-|\a|}
\]
for $x\in\re^n$ with some $C_\a>0$. 
\end{ass}

Under Assumption~\ref{ass0ba}, it is well-known that the wave operators:
\[
W_\pm =\slim_{t\to\pm\infty}e^{itH}e^{-itH_0}
\]
exist on $L^2(\re^n)$. 

\begin{thm}\label{thm0103a}
Suppose Assumption~\ref{ass0ba}, and suppose the global nontrapping condition. 
Then $W_\pm$ are FIOs associated to $w_\pm^{-1}$. 
\end{thm}

Note that $w_\pm$ is homogeneous in $\x$, and $W_\pm$ are FIOs 
in the sense of H\"ormander even if $\m<2$. 

\medskip

We prove our main results combining a Beals-type characterization of FIOs and 
Egorov-type theorems, which are variations of corresponding theorems proved in 
\cite{Na1, Na2, IN1} to characterize the wave front set of solutions to 
Schr\"odinger equations. We discuss the Beals-type characterization theorem in 
Section~2, and then its generalization to FIOs associated to asymptotically 
homogeneous canonical transforms in Section~4. Using these, we prove Theorem~\ref{thm0101a}
and Theorem~\ref{thm0102a} in Section~3 and Section~4, respectively. 
Application of these idea to wave operators (Theorem~\ref{thm0103a}) is discussed in 
Section~5. Several technical lemmas are proved in 
Appendix~A. 

\medskip

The fundamental solution to Schr\"odinger equation with the flat Laplacian as the principal terms 
has been studied by many authors, for example Fujiwara \cite{Fu} and Yajima~\cite{Y}. 
In this case, a global construction of the fundamental solution is known, and it was applied to 
various estimates, for example dispersive estimates of the Schr\"odinger evolution group. 

On the other hand, not much has been known about the fundamental solution to the 
Schr\"odinger equation with variable coefficients. The local regularity properties of the 
fundamental solution under nontrapping condition is known for some time, but it is not 
enough to characterize the singularities of solutions to the Schr\"odinger equation (cf., e.g., 
Kapitanski-Safarov \cite{KS}), because the equation has infinite propagation speed. 
The first step of the analysis of microlocal singularity for the equation was carried out by 
Craig, Kappeler and Strauss \cite{CKS}. They proved microlocal smoothing property of the equation, 
and thus gave a sufficient condition for the microlocal regularity of solutions. 
The result has been improved or generalized by Wunsch \cite{W}, Robbiano-Zuily \cite{RZ}, 
Nakamura \cite{Na0}, Ito \cite{I}, Martinez-Nakamura-Sordoni \cite{MNS1}, etc. 
Then a complete characterization of the microlocal singularities of solutions was 
given by Hassell and Wunsch \cite{HW1,HW2} by constructing a parametrix as 
a Legendre distribution on scattering manifolds. 
The result was later generalized by Nakamura \cite{Na1, Na2}, Ito-Nakamura \cite{IN1}
and Martinez-Nakamura-Sordoni \cite{MNS2}. In these papers, the authors do not construct a parametrix, 
but instead use Egorov-type theorems to obtain the characterization of singularities of solutions. 

In this paper, we show that these Egorov-type theorems actually imply that the fundamental 
solution is written in terms of  an FIO and the free evolution operator. 
Our result is analogous to results by Hassell and Wunsch \cite{HW1,HW2}, but the formulation
is quite different, and we work in the standard framework of FIOs and Lagrangian distributions, 
combined with classical mechanical scattering theory. Our model is restricted to Schr\"odinger 
operators on the Euclidean space, but our assumption on the perturbation is weaker than theirs. 
In particular, we consider general {\it short-range} perturbations, not necessarily having 
an asymptotic expansion starting from $O(|x|^{-2})$ terms.  We may also include operators 
with unbounded potential terms. 
Our method can be applied also to operators on scattering manifolds, and we discuss them 
in a forthcoming paper. 
\medskip

We use the following notation throughout this paper:
For function spaces $X$, $Y$, $\mathcal{L}(X,Y)$ denotes the linear space 
of the continuous linear maps from $X$ to $Y$.
We write $\re_+=(0,\infty)$ and $\re_-=(-\infty,0)$. Also we write
$\ze_+=\{0,1,2,\dots\}$.  
For $u\in \mathcal{S}(\re^n)$, we denote the Fourier transform by 
\[
\hat u(\x)=\mathcal{F}u(\x) = (2\pi)^{-n/2}\int_{\re^n} e^{-ix\cdot\x}f(x)\,dx.
\]
$\mathcal{F}$ is extended to a map from $\mathcal{S}'(\re^n)$ to $\mathcal{S}'(\re^n)$.
 We denote the standard symbol class of pseudodifferentail operators as follows: 
 We write $a\in S^m_{\rho,\d}(\re^n)$ if $a(x,\x)\in C^\infty(\re^n\times\re^n)$ and for any
 $\a,\b\in\ze_+^n$, 
 \[
 \bigabs{\pa_x^\a\pa_\x^\b a(x,\x)}\leq C_{\a\b}\jap{\x}^{m+\d|\a|-\rho|\b|}, \quad x,\x\in\re^n,
 \]
 with some $C_{\a\b}>0$. We write $a\in S_{cl}^m(\re^n)$ if $a\in S^m_{1,0}$ and $a$ has 
 an asymptotic expansion:
 \[
 a(x,\x)\sim \sum_{j=0}^\infty a_j(x,\x), \quad \text{as }|\x|\to\infty,
 \]
 where $a_j(x,\x)$ are homogeneous of order $(m-j)$ in $\x$. For a symbol $a(x,\x)$, 
 the pseudodifferential operator $a(x,D_x)$ is defined by 
 \[
 a(x,D_x)u(x) =(2\pi)^{-n/2}\int_{\re^n} e^{ix\cdot\x} a(x,\x) \hat u(\x)\,d\x,
 \quad u\in\mathcal{S}(\re^n), 
 \]
 and the Weyl quantization is defined by
 \[
 a^W(x,D_x)u(x) =(2\pi)^{-n}\int_{\re^{2n}} e^{i(x-y)\cdot\x} a\bigpare{\tfrac{x+y}{2},\x} 
 u(y)\,dy\,d\x.
\]
For $(x_0,\x_0)\in T^*M$, $\O\subset T^* M$ is called a {\it conic neighborhood}\/ of 
$(x_0,\x_0)$ if $\O$ is a neighborhood such that if $(x,\x)\in\O$ then $(x,\l \x)\in\O$ 
for any $\l\in[1,\infty)$. A conic set is called {\it narrow}\/ if for any $(x,\x), (x,\y)\in \O$, 
($\x,\y\neq 0$), $\hat \x\cdot\hat\y>a$ with fixed $a>0$, where $\hat \x=\x/|\x|$. 
 
%%%%%%%%%%%%%%%%%%%%%%%%%%%%%%%%%%%%%%%%%
%%%%%%%%%%%%  Section 2  %%%%%%%%%%%%%%%%
%%%%%%%%%%%%%%%%%%%%%%%%%%%%%%%%%%%%%%%%%

\section{Beals-type characterization of Fourier integral operators}

In this section we prove a Beals-type characterization theorem for Fourier 
integral operators. At first we recall several standard definitions. 

Let $\L\subset T^*\re^m$ be a smooth $m$-dimensional submanifold. 
$L$ is called {\it Lagrangian}\/ if the pull back of the standard canonical 
form vanishes on $\L$, i.e., $i^*(d\x\wedge dx)=0$ on $T^*\L$. 
$\L$ is called {\it conic}\/  if $(x,\x)\in\L$ implies $(x,\l \x)\in\L$ for $\l>0$. 

\begin{defn}[Besov space $B^{\s,\infty}_2(\re^n)$]\label{defn0201a}
Let $\s\in\re$ and let $u\in \mathcal{S}'(\re^m)$ such that $\hat u\in L^2_{loc}(\re^m)$. 
Then we set 
\[
\norm{u}_{B^{\s,\infty}_2}= \biggpare{\int_{|\x|\leq 1}|\hat u(\x)|^2 d\x}^{1/2}
+\sup_{j\geq 0} \biggpare{\int_{2^j\leq |\x|\leq 2^{j+1}} \bigabs{2^{\s j}\hat u(\x)}^2d\x}^{1/2}
\]
Then $B^{\s,\infty}_2(\re^m)$ is defined by
\[
B^{\s,\infty}_2(\re^m)=\bigset{u\in \mathcal{S}'(\re^m)}{\norm{u}_{B^{\s,\infty}_2}<\infty}
\]
and 
\[
B^{\s,\infty}_{2,loc}(\re^m) =\bigset{u\in\mathcal{D}'(\re^m)}{
\f u\in B^{\s,\infty}_2(\re^m) \text{ for any } \f\in C_0^\infty(\re^m)}.
\]
\end{defn}

According to H\"ormander \cite{Ho2} (see also Sogge \cite{So}), the Lagrangian 
distribution is defined as follows: 

\begin{defn}[Lagrangian distribution]\label{defn0202a}
Let $\L\subset T^*\re^m\setminus 0$ be a conic Lagrangian submanifold, 
$u\in \mathcal{S}'(\re^m)$ and $\s\in\re$. 
$u$ is called {\it Lagrangian distribution} associated to $\L$ of order $\s$ 
if for any $p_1,\dots, p_N\in S^1_{cl}(\re^m)$ such that 
the principal symbols of $p_j$ vanish on $\L$ ($j=1,2,\dots, N$), 
\[
p_1(x,D_x)p_2(x,D_x)\cdots p_N(x,D_x)u\in B^{-\s-m/4,\infty}_{2,loc}(\re^m), 
\]
and we denote $u\in I^\s(\re^m,\L)$.
\end{defn}

\begin{defn}[Fourier integral operators]\label{defn0203a}
Let $S$ be a canonical transform from $T^*\re^n$ to $T^*\re^n$, and suppose 
$S$ is homogeneous of order 1 with respect to $\x$. Let 
\[
\L_S=\bigset{(y,x,\y,-\x)}{(y,\y)=S(x,\x), (x,\x)\in T^*\re^n\setminus 0}\subset 
T^*\re^{2n}\setminus 0.
\]
Let $U\in \mathcal{L}(\mathcal{S}(\re^n), \mathcal{S}'(\re^n))$ and let 
$u\in \mathcal{S}'(\re^{2n})$ be its distribution kernel. 
Then $U$ is called a {\it Fourier integral operator}\/  associated to $S$ if 
$u\in I^\s(\re^{2n},\L_S)$. 
\end{defn}

Note $\L_S$ is a conic Lagrangian submanifold since $S$ is a homogeneous 
canonical transform.

\begin{rem}\label{rem0201a}
If $U$ is a Fourier integral operator, it is known that there is $m\leq 2n$, 
a phase function $\G(x,\th,y)$ ($x,y\in\re^n$, $\th\in\re^m$), which is homogeneous 
of order 1 in $\x$, and a symbol $a(x,\th,y)\in S^{\s+n/2-m/2}_{1,0}$ such that 
\[
U\f(x)=(2\pi)^{-n/2+m/2} \int_{\re^m\times\re^n} e^{i\G(x,\th,y)} a(x,\th,y)\f(y)dyd\th
\]
and we have a familiar representation of an FIO (see \cite{Ho2} or
\cite{So} for the detail). 
\end{rem}

We give a characterization of FIOs in terms of conjugation 
of operators. Let $S$ be a canonical transform and 
$U\in \mathcal{L}(\mathcal{S}(\re^n),\mathcal{S}'(\re^n))$ as above. 
Let $a\in S^m_{1,0}(\re^n)$ ($m\in\re$) such that 
\begin{equation}\label{eq0201a}
\bigset{x}{a(x,\x)\neq 0 \text{ for some }\x\in\re^n}\Subset \re^n, \quad 
\supp a\cap (\re^n\times\{0\})=\emptyset.
\end{equation}
For such $a$, we set
\[
Ad_S(a)U= (a\circ S^{-1})(x,D_x) U -U a(x,D_x)\ : \ \mathcal{S}\to \mathcal{S}'.
\]

\begin{thm}\label{thm0201a}
Let $S$ be as above, and let $U\in\mathcal{L}(\mathcal{S}(\re^n),\mathcal{S}'(\re^n))$. 
$U$ is an FIO of order 0 associated to $S$ if and only if for any 
$a_1, a_2,\dots,a_N\in S^1_{cl}(\re^n)$ satisfying \eqref{eq0201a}, 
\begin{equation}\label{eq0202a}
Ad_S(a_1)Ad_S(a_2)\cdots Ad_S(a_N)U\in \mathcal{L}\bigpare{L^2_{cpt}(\re^n),
L^2_{loc}(\re^n)}.
\end{equation}
\end{thm}

The next corollary gives convenient sufficient conditions. 

\begin{cor}\label{cor0202a}
(i) Let $S$ and $U$ be as in Theorem~\ref{thm0201a}. If for any $a\in S^1_{cl}(\re^n)$ 
satisfying \eqref{eq0201a}, there is $b\in S_{1,0}^0(\re^n)$ such that 
\[
Ad_S(a)U =b(x,D_x) U +R
\]
with a smoothing operator $R$, then $U$ is an FIO associated to $S$. \newline
(ii) Let $S$ and $U$ as above, and suppose $U$ is invertible. 
If for any $a\in S_{cl}^1(\re^n)$ satisfying \eqref{eq0201a}  there is $b\in S_{cl}^0(\re^n)$ 
such that 
\[
U a(x,D_x)U^{-1}=(a\circ S^{-1})(x,D_x) + b(x,D_x), 
\]
then $U$ is an FIO of order 0 associated to $S$. 
\end{cor}

\begin{proof} 
(i) The condition \eqref{eq0202a} with $N=1$ follows immediately from the assumption
of (i). Let $N=2$, and let $Ad_S(a_j)U=b_j(x,D_x)U +R_j$, $j=1,2$. Then we have 
\begin{align*}
Ad_S(a_1)Ad_S(a_2)U &= (a_1\circ S^{-1})(x,D_x) b_2(x,D_x) U
 -b_2(x,D_x) U a_1(x,D_x)\\
&\qquad+Ad_S(a_1) R_2\\
&= \bigbrac{(a_1\circ S^{-1})(x,D_x), b_2(x,D_x)}U-b_2(x,D_x)b_1(x,D_x) U \\
&\qquad  +Ad_S(a_1)R_2 + b_2(x,D_x)R_1 \\
&= b_{12}(x,D_x)U +R_{12}
\end{align*}
where $b_{12}\in S^0_{1,0}(\re^n)$ and $R_{12}$ is a smoothing operator. 
Repeating this procedure, we conclude \eqref{eq0202a} for any $N$. 
Now  (ii) follows easily from (i). 
\end{proof}

In order to prove Theorem~\ref{thm0201a}, we first notice that the $L^2_{cpt}$-$L^2_{loc}$ 
boundedness implies the distribution kernel is locally $B^{-n/2,\infty}_2$. 

\begin{lem}\label{lem0203a}
Suppose $U\in \mathcal{L}(L^2_{cpt}\bigpare{\re^n),L^2_{loc}(\re^n)}$, and let 
$u\in \mathcal{S}'(\re^n)$ be its distribution kernel. 
Then $u\in B^{-n/2,\infty}_{2,loc}(\re^n)$.
\end{lem}

\begin{proof}
Let $u$ be the distribution kernel of $U$. 
Let $\chi,\g\in C_0^\infty(\re^n;\re)$, and we suppose $\chi$ is a even function. 
We consider
\[
I= \int_{\re^{2n}}\bigabs{\g(\x)\g(\y)\mathcal{F}[\chi(x)\chi(y)u(x,y)](\x,\y)}^2 d\x\,d\y.
\]
We choose $\chi_1\in C_0^\infty(\re^n)$ so that $\chi_1(x)=1$ on $\supp\chi$. 
We note $I$ can be expressed in terms of the Hilbert-Schmidt norm: 
\[
I=\bignorm{\g\mathcal{F} \chi U \chi \mathcal{F}^{-1}\g}_{HS}^2
\]
where $\g=\g(\x)$ and $\chi=\chi(x)$ denote the multiplication operators 
on $L^2(\re^n_\x)$ and $L^2(\re^n_x)$, respectively. $\norm{\cdot}_{HS}$ 
denotes the Hilbert-Schmidt norm. Then we represent the Hilbert-Schmidt 
norm by a trace:
\begin{align*}
I &= \trace\bigbrac{(\g\mathcal{F}\chi U\chi\mathcal{F}^{-1}\g)^*
(\g\mathcal{F}\chi U\chi\mathcal{F}^{-1}\g)}\\
&= \trace\bigbrac{\g\mathcal{F}\chi U^*\chi\mathcal{F}^{-1}\g^2
\mathcal{F}\chi U\chi\mathcal{F}^{-1}\g}\\
&=\trace\bigbrac{((\chi_1U\chi_1)^*\chi\mathcal{F}^{-1}\g^2\mathcal{F}\chi)
((\chi_1U\chi_1)\chi\mathcal{F}^{-1}\g^2\mathcal{F}\chi)}\\
&= \trace\bigbrac{((\chi_1U\chi_1)^*\chi\g^2(D_x)\chi)((\chi_1U\chi_1)\chi\g^2(D_x)\chi)}
\end{align*}
We use the H\"older inequality for the trace to obtain 
\begin{align*}
I&\leq \bignorm{(\chi_1U\chi_1)^*\chi \g^2(D_x)\chi}_{HS}
\bignorm{(\chi_1U\chi_1) \chi\g^2(D_x)\chi}_{HS}\\
&\leq \norm{\chi_1U\chi_1}_{\mathcal{L}(L^2)}^2\bignorm{\chi\g^2(D_x)\chi}_{HS}^2\\
&=(2\pi)^{-n}\norm{\chi_1U\chi_1}_{\mathcal{L}(L^2)}^2 \norm{\g}_{L^4}^4\norm{\chi}_{L^2}^2
\norm{\chi}_{L^\infty}^2.
\end{align*}
Here we have used the well-known properties: 
$\norm{AB}_{HS}\leq \norm{A}\,\norm{B}_{HS}$ and 
$\norm{a(x)b(D_x)}_{HS}=(2\pi)^{-n/2}\norm{a}_{L^2}\norm{b}_{L^2}$. 

Now we choose $\g\in C_0^\infty(\re^n)$ so that $\g(\x)=1$ for $1\leq |\x|\leq 2$, 
and we set
\[
\g_N(\x)=\g(2^{-N}\x), \quad \text{for }N=1,2,\dots, \text{ and }\x\in\re^n.
\]
We note $\norm{\g_N}_{L^4}^4=2^{nN}\norm{\g}_{L^4}^4$. 
Then, by the above estimate, we have 
\begin{align*}
&\iint_{2^N\leq|\x|,|\y|\leq 2^{N+1}} \bigabs{\mathcal{F}[\chi(x)\chi(y)u(x,y)](\x,\y)}^2
d\x d\y\\
&\qquad \leq \iint \bigabs{\g_N(\x)\g_N(\y)\mathcal{F}[\chi(x)\chi(y)u(x,y)](\x,\y)}^2 d\x d\y\\
&\qquad \leq (2\pi)^{-n}\norm{\chi_1U\chi_1}_{\mathcal{L}(L^2)}^2 
\norm{\g}_{L^4}^4 \norm{\chi}_{L^2}^2\norm{\chi}_{L^\infty}^2\times 2^{nN},
\end{align*}
and this implies $\chi(x)\chi(y) u(x,y)\in B^{-n/2,\infty}_2(\re^{2n})$ for any 
$\chi\in C_0^\infty(\re^n)$. 
\end{proof}

We set 
\[
\tilde\L_S=\bigset{(y,\y,x,\x)}{(y,\y)=S(x,\x)}\subset T^*\re^n\times T^*\re^n.
\]

\begin{lem}\label{lem0204a}
Let $p\in S^1_{cl}(\re^{2n})$ such that the principal symbol of $p$ vanishes on $\tilde\L_S$, 
and suppose $p$ is supported in a narrow convex conic neighborhood of 
$(S(x_0,\x_0),x_0,\x_0)\in \tilde \L_S$. Then there exist $b_j\in S^0_{cl}(\re^{2n})$, 
$f_j\in S^1_{cl}(\re^n)$ $(j=1,2,\dots 2n)$, and $r\in S^0_{cl}(\re^{2n})$ such that 
\[
p(y,\y,x,\x)=\sum_{j=1}^{2n}b_j(y,\y,x,\x)
\bigpare{(f_j\circ S^{-1})(y,\y)-f_j(x,\x)}+r(y,\y,x,\x).
\]
\end{lem}

\begin{proof}
We may assume $p$ is homogeneous of order 1 without loss of generality. 
We denote 
\[
(z,\z)= S^{-1}(y,\y)
\]
so that $p(y,\y,z,\z)=0$. We let $\C_2, \C_3$ be convex conic neighborhoods of 
$(x_0,\x_0)$ such that $\overline{\C_2}\subset \C_3$, 
 and let $\C_0, \C_1$ be convex conic neighborhoods of $(y_0,\y_0,x_0,\x_0)$  such that
\[
\supp p \subset \C_0\subset\overline{\C_0}\subset \C_1\subset (S\C_2)\times\C_2.
\]
We choose $\chi\in S^0_{cl}(\re^{2n})$ so that $\chi=1$ on $\C_0$ and 
$\supp \chi\subset \C_1$. We also choose $\rho\in S_{cl}^0(\re^n)$ so that 
$\rho=1$ on $\C_2$ and $\supp\rho\subset \C_3$. Then we compute
\begin{align*}
p(y,\y,x,\x) &= p(y,\y,x,\x)-p(y,\y,z,\z) \\
&= \int_0^1 \frac{d}{dt}\bigpare{p(y,\y,tx+(1-t)z,t\x+(1-t)\z)}dt\\
&=\sum_{j=1}^n (x_j-z_j)\int_0^1 \frac{\pa p}{\pa x_j}(y,\y,tx+(1-t)z,t\x+(1-t)\z)dt \\
&\quad + \sum_{j=1}^n (\x_j-\z_j)\int_0^1 \frac{\pa p}{\pa \x_j}(y,\y,tx+(1-t)z,t\x+(1-t)\z)dt.
\end{align*}
We now set
\begin{align*}
& g_j(y,\y,x,\x) = \chi(y,\y,x,\x)\int_0^1 \frac{\pa p}{\pa x_j}(y,\y,tx+(1-t)z,t\x+(1-t)\z)dt\\
&g_{n+j}(y,\y,x,\x)=\chi(y,\y,x,\x)\int_0^1 \frac{\pa p}{\pa \x_j}(y,\y,tx+(1-t)z,t\x+(1-t)\z)dt
\end{align*}
for $j=1,2,\dots,n$. We note $g_j\in S^1_{cl}(\re^{2n})$ for $j=1,2,\dots, n$, and 
$g_j\in S^0_{cl}(\re^{2n})$ for $j=n+1,\dots, 2n$. By the choice of $\chi$, we have 
\begin{equation}\label{eq0203a}
p(y,\y,x,\x) =\sum_{j=1}^n (x_j-z_j)g_j(y,\y,x,\x)
+\sum_{j=1}^n (\x_j-\z_j)g_{n+j}(y,\y,x,\x). 
\end{equation}
We also set
\[
f_j(x,\x)=x_j|\x|\rho(x,\x), \quad f_{n+j}(x,\x)= \x_j\rho(x,\x)
\]
for $j=1,2,\dots, n$. Then, as well as the computation above,  we have 
\begin{align*}
&(f_j\circ S^{-1})(y,\y)-f_j(x,\x)
= f_j(z,\z)-f(x,\x) \\
&\quad =-\int_0^1 \frac{d}{dt} f_j(tx+(1-t)z,t\x+(1-t)\z)dt\\
&\quad = -\sum_{k=1}^n (x_k-z_k)\int_0^1 \frac{\pa f_j}{\pa x_k}(tx+(1-t)z,t\x+(1-t)\z)dt \\
&\quad \qquad - \sum_{k=1}^n (\x_k-\z_k)\int_0^1 \frac{\pa f_j}{\pa \x_k}(tx+(1-t)z,t\x+(1-t)\z)dt.
\end{align*}
It is easy to see that on $(S\C_2)\times \C_2$, we have 
\begin{align*}
&\frac{\pa f_j}{\pa x_k}(tx+(1-t)z,t\x+(1-t)\z) =\d_{jk}|t\x+(1-t)\z|, \\
&\frac{\pa f_j}{\pa \x_k}(tx+(1-t)z,t\x+(1-t)\z) =r_{jk}(y,\y,x,\x),\\
&\frac{\pa f_{n+j}}{\pa x_k}(tx+(1-t)z,t\x+(1-t)\z) =0,\\
&\frac{\pa f_{n+j}}{\pa \x_k}(tx+(1-t)z,t\x+(1-t)\z) =\d_{jk},
\end{align*}
where $r_{jk}\in S_{cl}^0(\re^{2n})$, $j,k=1,2,\dots, n$. 
Thus we have 
\begin{align*}
&(f_j\circ S^{-1})(y,\y)-f_j(x,\x) =-(x_j-z_j)\int_0^1 |t\x+(1-t)\z|dt +r_j(y,\y,x,\x)\\
& (f_{n+j}\circ S^{-1})(y,\x)-f_{n+j}(x,\x)=-(\x_j-\z_j)
\end{align*}
on $(S\C_2)\times \C_2$ for $j=1,2,\dots, n$, where $r_j\in S^0_{cl}(\re^{2n})$. Since $g_j$ are supported in 
$\C_1\subset (S\C_2)\times \C_2$, we can find $b_j\in S^0_{cl}(\re^{2n})$ such that 
\begin{align*}
&(x_j-z_j)g_j(y,\y,x,\x) = b_j(y,\y,x,\x) ((f_j\circ S^{-1})(y,\y)-f_j(x,\x))+r_j', \\
&(\x_j-\z_j)g_{n+j}(y,\y,x,\x)=b_{n+j}(y,\y,x,\x)((f_{n+j}\circ S^{-1})(y,\y)-f_{n+j}(x,\x))
\end{align*}
with $r_j'\in S^0_{cl}(\re^{2n})$, $j=1,2,\dots, n$. 
The assertion now follows from these and \eqref{eq0203a}. 
\end{proof}

\begin{proof}[Proof of Theorem 2.1]
The ``only if '' part is straightforward: If $a_j\in S^1_{cl}(\re^n)$ satisfying \eqref{eq0201a}, 
then $p_j(y,\y,x,\x)=(a_j\circ S^{-1})(y,\y)-a_j(x,-\x)$ vanish on $\L_S$, and 
hence \eqref{eq0202a} follows from the definition of the FIOs and the $L^2$-boundedness 
theorem of FIOs (see, e.g., \cite{Ho2} Theorem~25.3.1 or \cite{So} Theorem~6.2.1). 

We suppose \eqref{eq0202a} and show the ``if'' part. At first, we note $\WF(u)\subset \L_S$: 
If $(y_0,x_0,\y_0,-\x_0)\notin \L_S$, then we can find 
$a\in S_{cl}^0(\re^n)$ and $b\in S^1_{cl}(\re^n)$ such that $a_0(y_0,\y_0)\neq 0$, 
$b_0(x_0,\x_0)\neq 0$ and that $a$ and $b$ are supported in  small conic 
neighborhoods of $(y_0,\y_0)$ and $(x_0,\x_0)$, respectively, so that 
$a(y,\y)\cdot (b\circ S^{-1})(y,\y)=0$. By \eqref{eq0202a} with Lemma~\ref{lem0203a}, we learn 
\begin{align*}
&a(y,D_y)((b\circ S^{-1})(y,D_y)-b(x,-D_x))u \\
&\qquad =(-a(y,D_y)b(x,-D_x)+R(y,D_y))u 
\in B^{-n/2,\infty}_{2}(\re^{2n})
\end{align*}
with $R\in S^0_{1,0}(\re^{n})$. 
This implies $a(y,D_y)b(x,-D_x)u\in B^{-n/2,\infty}_{2}(\re^{2n})$ by the boundedness 
of $R$ in $B^{-n/2,\infty}_2(\re^{2n})$ (see \cite{Ho2} Corollary~B.1.6). Iterating this procedure, we learn 
\[
[a(y,D_y)b(x,-D_x)]^N u\in B^{-n/2,\infty}_2(\re^{2n})
\]
for any $N$, and this implies $(y_0,\y_0,x_0,-\x_0)\notin \WF(u)$. 

We now let $p_1,p_2,\dots, p_N\in S^1_{cl}(\re^{2n})$ such that $p_j$ vanish
on $\tilde \L_S$, and we show 
\[
p_1(y,D_y,x,-D_x)p_2(y,D_y,x,-D_x)\cdots p_N(y,D_y,x,-D_x)u
\in B^{-n/2,\infty}_{2,loc}(\re^{2n}).
\]
By the above observation, we may assume $u$ is essentially supported in
an arbitrarily small conic neighborhood of $\L_S$. Moreover, by partition of unity, 
we may also assume $p_j$ are supported in a small convex conic neighborhood 
of $(y_0,\y_0,x_0,\x_0)\in \tilde \L_S$, where $(y_0,\y_0)=S(x_0,\x_0)$.
Then by Lemma\ref{lem0204a}, we have 
\begin{align*}
p_j(y,D_y,x,-D_x) &= \sum_{k=1}^{2n} b_{jk}(y,D_y,x,-D_x)((f_k\circ S^{-1})(y,D_y)
-f_k(x,-D_x)) \\
& \quad +r_j(y,D_y,x,-D_x)
\end{align*}
for each $j=1,\dots, N$, where $b_{jk}\in S^0_{cl}(\re^{2n})$ and 
$f_j\in S^1_{cl}(\re^{2n})$ are those given in Lemma\ref{lem0204a}, and 
$r_j\in S(1,dy^2+\frac{d\y^2}{\jap{\y}^2}+dx^2+\frac{d\x^2}{\jap{\x}^2})$, 
where $S(m,g)$ denotes the symbol class defined in \cite{Ho2}, Section~18.5 (Weyl calculus). 
Then by simple symbol calculus, we can show 
\begin{align}
&\prod_{j=1}^N p_j(y,D_y,x,-D_x)u 
= \sum_{k_1=1}^{2n}\cdots \sum_{k_N=1}^{2n}\prod_{j=1}^N b_{jk_j}(y,D_y,x,-D_x)\times \label{eq0204a}\\
&\quad \times \prod_{j=1}^N \bigpare{(f_{k_j}\circ S^{-1})(y,D_y)-f_{k_j}(x,-D_x)}u 
+R(y,D_y,x,-D_x)u \nonumber
\end{align}
with some $R\in S(1,dy^2+\frac{d\y^2}{\jap{\y}^2}+dx^2+\frac{d\x^2}{\jap{\x}^2})$. 
Note $R$ is bounded in $B^{-n/2,\infty}_2(\re^{2n})$. 
Each term in the right hand side has the form 
\[
B(y,D_y,x,-D_x) \, \mbox{ker}\bigbrac{Ad_S(f_{k_1})\cdots Ad_S(f_{k_N})U}
\]
with $B\in S_{cl}^0(\re^{2n})$ except for $R$, where $\mbox{ker}[A]$ denotes the distribution 
kernel of an operator $A$. Now the claim follows from the assumption and Lemma~\ref{lem0203a}. 
\end{proof}

%%%%%%%%%%%%%%%%%%%%%%%%%%%%%%%%%%%%%%%%%
%%%%%%%%%%%%  Section 3  %%%%%%%%%%%%%%%%
%%%%%%%%%%%%%%%%%%%%%%%%%%%%%%%%%%%%%%%%%

\section{Proof of Theorem~\ref{thm0101a}}

Here we  prove that, under Assumption~\ref{ass0aa} with $\m=2$, $W(t)=e^{itH_0}e^{-itH}$ 
satisfies the condition of Corollary~\ref{cor0202a}-(ii) with $S=w_\pm$, where $t\in\re_\pm$. 
The condition is an Egorov-type theorem, and it was essentially proved in 
\cite{Na1} (see also \cite{Na2,IN1}) in the semiclassical formalism. 
We modify the argument to prove the Egorov theorem in $S^1_{1,0}(\re^n)$ 
symbol class. Namely, we prove the following:

\begin{thm}\label{thm0301a}
Suppose Assumption~\ref{ass0aa} with $\m=2$, and suppose the global nontrapping condition. 
Let $\pm t>0$. 
Then for any $a\in S^1_{1,0}(\re^n)$ satisfying \eqref{eq0201a}, there is $b\in S^0_{1,0}(\re^n)$ 
such that 
\[
W(t) a(x,D_x) W(t)^{-1} = (a\circ w_\pm^{-1})(x,D_x) +b(x,D_x).
\]
\end{thm}

We first sketch the outline following \cite{Na1}. We set
\[
A(t)= W(t) a(x,D_x)W(t)^{-1}, \quad t\in\re.
\]
If we consider $W(t)$ as an evolution operator, we can compute the generator as follows: 
For $\g\in \mathcal{S}(\re^n)$, we have 
\begin{align*}
\frac{d}{dt} W(t)\g &= e^{itH_0} (iH_0-iH) e^{-itH_0}W(t)\g \\
&= -i \bigpare{ e^{itH_0} H e^{-itH_0} -H_0} W(t)\g =-iL(t) W(t)\g. 
\end{align*}
Since 
\[
e^{itH_0} D_x e^{-itH_0} =D_x, \quad e^{itH_0} x e^{-itH_0} = x-tD_x,
\]
we learn that the principal symbol of $L(t)$ is given by 
\[
\ell(t,x,\x) = \frac12 \sum_{j,k=1}^n (a_{jk}(x-t\x)-\d_{jk})\x_j\x_k +V(x-t\x).
\]
In fact, if we use the Weyl calculus, which we do, the symbol of $L(t)$ is given by 
$\ell(t,x,\x)$ modulo $S^0_{1,0}$-terms. The classical flow generated by $\ell(t,x,\x)$ is 
\[
w(t)=\exp(tH_{p_0})\circ \exp(tH_p).
\]
Here we use, however, the flow:
\[
w_0(t) =\exp(tH_{p_0})\circ \exp(tH_k)
\]
which is generated by 
\[
\ell_0(t,x,\x) = \frac12 \sum_{j,k=1}^n (a_{jk}(x-t\x)-\d_{jk})\x_j\x_k. 
\]
Analogously to the usual Egorov theorem, we expect the principal symbol of $A(t)$ is 
given by $(a_0\circ w_0(t)^{-1})(x,\x)$,where $a_0$ is the principal symbol of $a$. 
We construct an asymptotic expansion of $A(t)$ by solving transport equations iteratively. 
We set
\[
\g_0(t,x,\x)= (a\circ w_0(t)^{-1})(x,\x)
\]
for $a\in S^1_{1,0}(\re^n)$. We note that by Lemma~\ref{lem0a01a}, 
$\g_0(t,\cdot,\cdot)\in S^1_{1,0}(\re^n)$, uniformly in $t$. For a symbol 
$q(x,\x)\in S^m_{1,0}(\re^n)$, we define a family of seminorms by 
\[
|q|_{m,L,K} =\sum_{|\a|+|\b|\leq L}\sup_{x\in K, \x\in\re^n} 
\bigabs{\jap{\x}^{-m+|\b|}\pa_x^\a\pa_\x^\b q(x,\x)}
\]
for $L\in\mathbb{N}$, $K\Subset \re^n$. For $T>0$, we write $I_T=[-T,T]$. 

\begin{lem}\label{lem0302a}
Let $a\in S^1_{1,0}(\re^n)$ satisfying \eqref{eq0201a}. Then there exists $\g(t,x,\x)$ such that 
\begin{enumerate}
\renewcommand{\theenumi}{\roman{enumi}}
\renewcommand{\labelenumi}{{{\rm (\theenumi)} }}
\item $\g(0,x,\x)=a(x,\x)$.
\item $\g(t,\cdot,\cdot)\in S^1_{1,0}(\re^n)$ and for any $L,T>0$ and  $K\Subset \re^n$, 
\[
|\g(t,\cdot,\cdot)|_{1,L,K}\leq C_{L,T,K}, \quad t\in I_T. 
\]
\item $\g(t,\cdot,\cdot)-\g_0(t,\cdot,\cdot)\in S^0_{1,0}(\re^n)$, and for any 
$L,T>0$ and  $K\Subset \re^n$, 
\[
|\g(t,\cdot,\cdot)-\g_0(t,\cdot,\cdot)|_{0,L,K}\leq C_{L,TK}, \quad t\in I_T. 
\]
\item Let $G(t)=\g^W\!(t,x,D_x)$, and set
\[
R(t)=\frac{d}{dt} G(t)-i[L(t),G(t)].
\]
Then $R(t)$ is a smoothing operator, and 
$\norm{\jap{D_x}^N R(t)\jap{D_x}^N}_{\mathcal{L}(L^2)}\leq C_{T,N}$, 
 for any $N$, $t\in I_T$. 
\end{enumerate}
\end{lem}

\begin{proof}
We can find $K\Subset \re^n$ such that $\g_0(t,x,\x)=0$ if $x\notin K$, since 
$w_0(t)(x,\x)$ has limit as $t\to\pm\infty$. We note 
$\ell_0(t,\cdot,\cdot)\in S^2_{1,0}(\re^n)$ and for any $L$, $|\ell_0(t,\cdot,\cdot)|_{2,L,K}$ 
is uniformly bounded. By the construction, $\g_0$ satisfies
\[
\frac{\pa}{\pa t}\g_0(t,x,\x)=-\{\ell_0,\g_0\}(t,x,\x),
\]
where $\{\cdot,\cdot\}$ denotes the Poisson bracket. Then by virtue of the Weyl calculus, 
we learn 
\[
\frac{\pa}{\pa t} \g_0^W\!(t,x,D_x)+i[L(t),\g_0^W(t,x,D_x)]=r_0^W\!(t,x,D_x)
\]
with $r_0\in S^0_{1,0}(\re^n)$, and the seminorms of $r_0$ 
are locally uniformly bounded in $t$. Then we solve the transport equation:
\begin{equation}\label{eq0301a}
\frac{\pa}{\pa t} \g_1(t,x,\x)+\{\ell_0,\g_1\}(t,x,\x)=-r_0(t,x,\x)
\end{equation}
with $\g_0(0,x,\x)=0$. It is easy to see that $\g_1\in S^0_{1,0}(\re^n)$ and 
the seminorms are locally uniformly bounded in $t$. Iterating this procedure, 
we obtain $\g_j\in S^{1-j}_{1,0}(\re^n)$, $j=1,2,\dots$, and we set the 
asymptotic sum as $\g$: 
\[
\g\sim \sum_{j=0}^\infty \g_j \quad \text{in }S^1_{1,0}(\re^n).
\]
Now it follows from the above construction that 
\begin{equation}\label{eq0302a}
\frac{\pa}{\pa t}\g^W\!(t,x,D_x) +i[L(t),\g^W\!(t,x,D_x)]=r(t,x,D_x)
\end{equation}
with $r\in S^{-\infty}_{1,0}(\re^n)$. Thus our $\g$ satisfies the required properties. 
\end{proof}

\begin{proof}[Proof of Theorem 3.1]
Let $\g(t,x,\x)$ be as in Lemma~\ref{lem0302a} and let $G(t)=\g^W\!(t,x,D_x)$. 
By the lemma, we have 
\[
\frac{d}{dt}(W(t)^{-1}G(t) W(t)) = W(t)^{-1}R(t) W(t)
\]
and the RHS is a smoothing operator, and its seminorms are uniformly bounded. 
Since $W(0)^{-1}G(0)W(0)=a^W\!(x,D_x)$, we learn 
\begin{equation}\label{eq0303a}
W(t)a^W\!(x,D_x)W(t)^{-1} -G(t) =R_2(t)
\end{equation}
is a smoothing operator. Thus, the principal symbol of $A(t)$ is 
$a\circ w_0(t)^{-1}$. It remains to compare $a\circ w_0(t)^{-1}$ with $a\circ w_\pm^{-1}$. 

We denote 
\[
(\tilde x(t,z,\y),\tilde\x(t,z,\y))= w_0(t)^{-1}(z,\y), \quad 
(\tilde x_\pm(z,\y),\tilde \x_\pm(z,\y))=w_\pm^{-1}(z,\y).
\]
Then by Lemma~\ref{lem0a01a}, we learn 
\begin{align*}
&\bigabs{\pa_z^\a\pa_\y^\b (\tilde x(t,z,\y)-\tilde x_\pm(z,\y))}
\leq C_{\a\b}\jap{\y}^{-|\b|}\jap{t\y}^{-\m+1}, \\
&\bigabs{\pa_z^\a\pa_\y^\b (\tilde \x(t,z,\y)-\tilde \x_\pm(z,\y))}
\leq C_{\a\b}\jap{\y}^{1-|\b|}\jap{t\y}^{-\m+1}
\end{align*}
for $\pm t>0$. We then compute 
\begin{align*}
&(a\circ w_0(t)^{-1})(z,\y) - (a\circ w_\pm^{-1})(z,\y) \\
&\quad = a(\tilde x(t,z,\y),\tilde\x(t,z,\y))-a(\tilde x_\pm(z,\y),\tilde\x_\pm(z,\y))\\
&\quad =\int_0^1 \frac{\pa}{\pa s}a(s\tilde x(t,z,\y)+(1-s)\tilde x_\pm(z,\y),
s\tilde \x(t,z,\y)+(1-s)\tilde \x_\pm(z,\y))ds\\
&\quad =(\tilde x(t,z,\y)-\tilde x_\pm(z,\y))\int_0^1(\pa_x a)(s\tilde x+(1-s)\tilde x_\pm,
s\tilde \x+(1-s)\tilde \x_\pm)ds\\
&\qquad +(\tilde \x(t,z,\y)-\tilde \x_\pm(z,\y))\int_0^1(\pa_\x a)(s\tilde x+(1-s)\tilde x_\pm,
s\tilde \x+(1-s)\tilde \x_\pm)ds\\
&\quad =(\tilde x(t)-\tilde x_\pm)\cdot A(z,\y)+(\tilde\x(t)-\tilde\x_\pm)\cdot B(z,\y),
\end{align*}
and it is easy to see
\[
\bigabs{\pa_z^\a\pa_\y^\b A(z,\y)}\leq C_{\a\b}\jap{\y}^{1-|\b|}, 
\quad \bigabs{\pa_z^\a\pa_\y^\b B(z,\y)}\leq C_{\a\b}\jap{\y}^{-|\b|}.
\]
Combining these, we now have 
\[
\bigabs{\pa_z^\a\pa_\y^\b(a\circ w_0(t)^{-1}-a\circ w_\pm^{-1})} 
\leq C_{\a\b}\jap{\y}^{1-|\b|}\jap{t\y}^{-\m+1}
\]
for $\pm t>0$. For fixed $t\neq 0$, this implies 
$a\circ w_0(t)^{-1}-a\circ w_\pm^{-1}\in S^{2-\m}_{1,0}=S^0_{1,0}$ since 
$\m=2$. Combining this with $\g(t,\cdot,\cdot)-a\circ w_0(t)^{-1}\in S^0_{1,0}$, we 
learn $\g(t,\cdot,\cdot)-a\circ w_\pm\in S^0_{1,0}$. 
The assertion follows from this and \eqref{eq0303a}.  
\end{proof}

Theorem~\ref{thm0101a} now follows immediately from Theorem~\ref{thm0301a} and 
Corollary~\ref{cor0202a}-(ii). 
\qed

%%%%%%%%%%%%%%%%%%%%%%%%%%%%%%%%%%%%%%%%%
%%%%%%%%%%%%  Section 4  %%%%%%%%%%%%%%%%
%%%%%%%%%%%%%%%%%%%%%%%%%%%%%%%%%%%%%%%%%

\section{Fourier integral operators associated to asymptotically homogeneous canonical 
transform}

Here we discuss FIOs associated to asymptotically homogeneous canonical transform, 
e.g., $w(t)$ and $w_0(t)$, and we prove Theorem~\ref{thm0102a}. We start with several definitions.

\begin{defn}\label{defn0401a}
Let $\L\subset T^*\re^m\setminus 0$ be a $d$-dimensional conic submanifold, and let 
$(x_0,\x_0)\in \L$. Suppose $\O$ be a conic neighborhood of $(x_0,\x_0)$ and let 
$\F:\O\to \re^{2m}$ be a local coordinate system on $\O$. $\F$ is called an 
{\it admissible conic local coordinate system}\/ (associated to $\L$) if 
$\F$ satisfies the following conditions:
\begin{enumerate}
\renewcommand{\theenumi}{\roman{enumi}}
\renewcommand{\labelenumi}{{{\rm (\theenumi)} }}
\item $\F$ is expressed as 
\[
\F(x,\x)=(|\x|, \s(x,\hat \x),\t(x,\hat\x)), \quad \s(x,\hat\x)\in\re^{d-1}, \t(x,\hat\x)\in\re^{2m-d}
\]
for $(x,\x)\in T^*\re^{2m}$, where $\hat \x=\x/|\x|$, i.e., $\s(x,\x)$ and $\t(x,\x)$ are independent of $|\x|$.
\item $\L\cap\O =\bigset{(x,\x)\in\O}{\t(x,\x)=0}$. 
\end{enumerate}
\end{defn}

\begin{defn}\label{defn0402a}
Let $\L\subset T^*\re^m$ be a $d$-dimensional submanifold. $\L$ is called {\it asymptotically 
conic}\/ if $\L$ satisfies the following conditions: 
\begin{enumerate}
\renewcommand{\theenumi}{\roman{enumi}}
\renewcommand{\labelenumi}{{{\rm (\theenumi)} }}
\item There exists a $d$-dimensional conic submanifold $\L_c\subset T^*\re^m$ such that for any 
$K\Subset \re^m$ and $\O\subset T^*\re^m$ : a conic neighborhood of $\L_c\cap (K\times\re^m)$, 
there is $R>0$ such that 
\[
\L\cap \bigset{(x,\x)}{x\in K, |\x|\geq R}\subset \O.
\]
\item Let $\O$ be a conic neighborhood of $(x_0,\x_0)\in\L_c$, and let 
$\F$ be an admissible conic local coordinate system on $\O$. 
Then there are $R>0$,  an $\re^{2m-d}$-valued function $\f(\l,\s)$ such that 
\begin{multline*}
\L\cap \bigset{(x,\x)\in\O}{|\x|\geq R} \\
=\bigset{(x,\x)\in\O}{\t(x,\x)=\f(|\x|,\s(x,\x)), |\x|\geq R},
\end{multline*}
and $\e>0$ such that for any $k\in\ze_+$ and $\a\in\ze_+^{d-1}$, 
\[
\bigabs{\pa_\l^k\pa_\s^\a \f(\l,\s)}\leq C_{k\a}\jap{\l}^{-\e-k}, \quad \l\geq R.
\]
\end{enumerate}
\end{defn}

\begin{defn}[Lagrangian distribution]\label{defn0403a}
Let $\L\subset T^*\re^m$ be an asymptotically conic Lagrangian submanifold, and let 
$u\in \mathcal{S}'(\re^m)$. Let $\n\in\re$. $u$ is called a Lagrangian distribution associated 
to $\L$ of order $\n$, if for any $p_1,\dots,p_N\in S^1_{1,0}(\re^m)$ such that $p_j=0$ on 
$\L$ ($j=1,2,\dots, N$), 
\[
p_1(x,D_x)p_2(x,D_x)\cdots p_N(x,D_x)u\in B^{-\n-m/4,\infty}_{2,loc}(\re^m). 
\]
We then write $u\in I^\n(\L,\re^m)$. 
\end{defn}

\begin{defn}\label{defn0404a}
Let $S: T^*\re^n \to T^*\re^n$ be a diffeomorphism. $S$ is called {\it asymptotically 
homogeneous} (of order one) if $S$ satisfies the following conditions: 
We write $S(x,\x)=(y(x,\x),\y(x,\x))$. There exists $S_c: T^*\re^n\to T^*\re^n$, a homogeneous 
canonical map (of order one) such that 
\[
y(x,\x)-y_c(x,\x)\in (S^{-\e}_{1,0}(\re^n))^n, \quad 
\y(x,\x)-\y_c(x,\x)\in (S^{1-\e}_{1,0}(\re^n))^n, 
\]
with some $\e>0$, where we denote $S_c(x,\x)=(y_c(x,\x),\y_c(x,\x))$.
\end{defn}

\begin{rem}\label{rem0401a}
By Lemmas~A.1 and A.2, we learn $w_0(t)$ and $w(t)$ are asymptotically homogeneous 
for $t\neq 0$, and they are associated to homogeneous canonical transforms $w_\pm$ 
for $t\in\re_\pm$, respectively. 
\end{rem}

\begin{lem}\label{lem0401a}
Suppose $S: T^*\re^n\to T^*\re^n$ is an asymptotically homogeneous canonical transform. 
Then 
\[
\L_S=\bigset{(y,x,\y,-\x)\in T^*\re^{2n}}{(y,\y)=S(x,\x)}
\]
is an asymptotically conic Lagrangian submanifold of $T^*\re^{2n}$. 
\end{lem}

\begin{proof}
Let $S_c$ be the associated homogeneous canonical transform, and let $\L_{S_c}$ 
be the corresponding Lagrangian manifold. Let $(x_0,\x_0)\in T^*\re^n$ and let 
$(y_0,x_0,\y_0,-\x_0)\in \L_{S_c}$ with $(y_0,\y_0)=S_c(x_0,x_0)$. 
Let $\O_1\times \O_2\subset \re^n\times S^{n-1}$ be a small neighborhood of 
$(x_0,\hat\x_0)\in\L_{S_c}$ and let 
\[
\g:\O_1\times\O_2\to\re^n\times\re^{n-1}, \quad (x,\o)\mapsto (x-x_0, \s(\o))
\]
be a local coordinate system on $\O_1\times\O_2$. We then set
\[
\G:\ \re_+\times \g(\O_1\times\O_2)\times B_\e(0)\times B_\e(0) \to T^*\re^{2n}
\]
with 
\begin{multline*}
\G:\ (\l,\a,\b,\t,\t')\mapsto\\ 
 (y_c(\g^{-1}(\a,\b))+\t, 
x_0+\a,\l(\y_c(\g^{-1}(\a,\b))+\t'),-\l\s^{-1}(\b)),
\end{multline*}
where $B_\e(0)=\{x\in\re^n\,|\, |x|<\e\}$ with sufficiently small $\e>0$. 
Then $\Ran\G$ is a conic neighborhood of $(y_0,x_0,\y_0,-\x_0)$ 
in $T^*\re^{2n}$, and $\G^{-1}$ is an admissible conic local coordinate system. 
We note if $R$ is sufficiently large, 
\begin{align*}
&\L_S\cap (\Ran \G(\{\l>R\})) \\
&\quad =\bigset{\bigpare{y(x_0+\a,\l\s^{-1}(\b)),x_0+\a,\y(x_0+\a,\l\s^{-1}(\b)),-\l\s^{-1}(\b)}}
{\\
&\hspace{7cm} \l>R,(\a,\b)\in\g(\O_1\times\O_2)}
\end{align*}
and hence $(y,x,\y,-\x)\in \L_S\cap (\Ran\G(\{\l>R\}))$ if and only if 
\begin{align*}
&\t=y(x_0+\a,\l\s^{-1}(\b))-y_c(\g^{-1}(\a,\b)), \\
&\t'=\l^{-1}(\y(x_0+\a,\l\s^{-1}(\b))-\y_c(x_0+\a,\l\s^{-1}(\b))),
\end{align*}
where $(\l,\a,\b,\t,\t')=\G^{-1}(y,x,\y,-\x)$. Now it is easy to check $\L_S$ satisfies conditions
of Definition~\ref{defn0402a} if we set $\s\to (\a,\b)$, $\t\to(\t,\t')$, 
\[
\f_j(\l,\a,\b)= y_j(x_0+\a,\l \s^{-1}(\b))-y_{c,j}(x_0+\a,\l\s^{-1}(\b)), 
\]
and 
\[
\f_{n+j}(\l,\a,\b)= \l^{-1}\bigpare{\y_j(x_0+\a,\l\s^{-1}(\b))-\y_{c,j}(x_0+\a,\l\s^{-1}(\b))}
\]
for $j=1,2,\dots,n$. 
\end{proof}

\begin{defn}\label{defn0405a}
Let $S$ be an asymptotically homogeneous canonical transform from $T^*\re^n$ to $T^*\re^n$, 
and let $\L_S$ be the associated Lagrangian manifold in $T^*\re^{2n}$, which is asymptotically 
conic by Lemma\ref{lem0401a}. Let $U\in \mathcal{L}(\mathcal{S}(\re^n),\mathcal{S}'(\re^n))$ and let
$u\in \mathcal{S}'(\re^{2n})$ be its distribution kernel. Then $U$ is called a Fourier integral 
operator associated to $S$ of order $\s\in\re$ if $u\in I^\s(\L_S,\re^{2n})$.
\end{defn}

Given Definition~\ref{defn0405a}, we now have the exact meaning of Theorem~\ref{thm0102a}, and we prove 
Theorem~\ref{thm0102a} in the remaining of this section. We note Theorem~\ref{thm0201a} holds with little 
modification for FIOs associated to asymptotically homogeneous canonical transforms: 

\begin{thm}\label{thm0402a}
Let $S: T^*\re^n\to T^*\re^n$ be an asymptotically homogeneous canonical transform, 
and let $U\in \mathcal{L}(\mathcal{S}(\re^n),\mathcal{S}'(\re^n))$. $U$ is an FIO associated 
to $S$ if and only if for any $a_1,a_2,\dots, a_N\in S^1_{1,0}(\re^n)$ satisfying \eqref{eq0201a}, 
\[
Ad_S(a_1)Ad_S(a_2)\cdots Ad_S(a_N)U \in \mathcal{L}(L^2_{cpt}(\re^n),L^2_{loc}(\re^n)).
\]
\end{thm}

The proof of Theorem~\ref{thm0402a} is almost the same as that of Theorem~\ref{thm0201a}. In the ``only if'' part, 
we use the fact that $L^2$-boundedness theorem holds for FIOs associated to 
asymptotically homogeneous canonical transforms. An analogue of Lemma\ref{lem0204a} is given 
as follows:

\begin{lem}\label{lem0403a}
Let $S$ be an asymptotically homogeneous canonical transform as above. 
Let $p\in S^1_{1,0}(\re^{2n})$ such that $p$ vanishes on 
\[
\tilde\L_S=\bigset{(y,\y,x,\x)}{(y,\y)=S(x,\x)}.
\]
Then there exist
$b_j\in S^0_{1,0}(\re^{2n})$, $f_j\in S^1_{1,0}(\re^n)$, $(j=1,2,\dots, 2n)$, and 
$r\in S^0_{1,0}(\re^{2n})$ such that 
\[
p(y,\y,x,\x)=\sum_{j=1}^{2n}b_j(y,\y,x,\x)
\bigpare{(f_j\circ S^{-1})(y,\y)-f_j(x,\x)}+r(y,\y,x,\x).
\]
\end{lem}

Lemma~\ref{lem0403a} is proved in almost the same manner as Lemma~\ref{lem0204a}.
Then the rest of the proof of Theorem~\ref{thm0402a} is just follows from 
the argument of Theorem~\ref{thm0201a}, and we omit the detail. 

\begin{proof}[Proof of Theorem~\ref{thm0102a}]
Given the above formulation, the proof of Theorem~\ref{thm0402a} is similar to that of Theorem~\ref{thm0101a}.
Here we explain only the necessary modifications. When we construct the asymptotic solution 
to the Heisenberg equation: $\pa_t G(t)=i[L(t),G(t)]$, we use 
\[
\g_0(t,x,\x)= (a\circ w(t)^{-1})(x,\x)
\]
instead of $(a\circ w_0(t)^{-1})(x,\x)$ in Section~3. By virtue of Lemma~\ref{lem0a02a}, we learn 
$\g_0\in S^1_{1,0}(\re^n)$, and we can carry out the symbol calculus as in Section~3 
with no difficulty. Then the remainder terms of the asymptotic expansion (e.g., $r_0$ in 
the proof of Lemma~\ref{lem0302a}) is in $S^0_{1,0}(\re^n)$ even if $1<\m<2$, since $w(t)$ includes 
the influence of the potential function $V(x)$. The rest of the proof is almost identical. 
\end{proof}

%%%%%%%%%%%%%%%%%%%%%%%%%%%%%%%%%%%%%%%%%
%%%%%%%%%%%%  Section 5  %%%%%%%%%%%%%%%%
%%%%%%%%%%%%%%%%%%%%%%%%%%%%%%%%%%%%%%%%%

\section{Microlocal structure of wave operators}

Throughout this section, we suppose Assumption~\ref{ass0ba} with $1<\m<2$, and we prove 
Theorem~\ref{thm0103a}.
We use an argument analogous to Lemma~\ref{lem0302a}, but we need to examine the $t$-dependence 
of seminorms more carefully. We note 
\begin{align}
&\bigabs{\pa_x^\a\pa_\x^\b \ell_0(t,x,\x)}\leq C_{\a\b K}\jap{t\x}^{-\m-|\a|}\jap{\x}^{2-|\b|}, \label{eq0501a}\\
&\bigabs{\pa_x^\a\pa_\x^\b \tilde V(t,x,\x)}\leq C_{\a\b K}\jap{t\x}^{-\m-|\a|}\jap{\x}^{-|\b|}\label{eq0502a}
\end{align}
for any $\a,\b\in\ze_+^n$, $K\Subset \re^n$, where
$\tilde V(t,x,\x)=V(x-t\x)$.
For $a\in S^1_{cl}(\re^n)$ satisfying \eqref{eq0201a},  
we set $\g_0(t,x,\x)=(a\circ w_0(t)^{-1})(x,\x)$ as in Section~3. Then we have 
\[
\bigabs{\g_0(t,\cdot,\cdot)}_{1,L,K}\leq C_{L,K}, \quad t\in\re,
\]
for any $L>0$, $K\Subset\re^n$. Moreover, by \eqref{eq0501a} and \eqref{eq0502a}, we also have 
\[
\bigabs{r_0(t,\cdot,\cdot)}_{0,L,K}\leq C_{L,K}\jap{t}^{-\m},\quad t\in\re, 
\]
where $r_0$ is defined as in the proof of Lemma~\ref{lem0302a}. 
Here we have used the fact $|\x|\geq c>0$ for all $t$ on the support of $\g_0$. 
Thus, since the solution to the transport equation \eqref{eq0301a} is uniformly bounded, 
we have 
\[
\bigabs{\g_1(t,\cdot,\cdot)}_{0,L,K}\leq C_{L,K}, \quad t\in\re.
\]
We repeat this procedure. We set $\g_j$ be the solution to the transport equations:
\[
\frac{\pa}{\pa t}\g_j(t,x,\x) +\{\ell_0,\g_j\}(t,x,\x)=-r_{j-1}(t,x,\x)
\]
with $\g_j(0,x,\x)=0$ (as in the proof of lemma~3.2), and we set 
$r_j\in S^{-j}_{1,0}$ such that 
\[
r_j^W\!(t,x,D_x) =\frac{\pa}{\pa t}\g_j^W\!(t,x,D_x)+i[L(t),\g_j^W\!(t,x,D_x)]+r_{j-1}^W(t,x,D_x)
\]
given $\g_j\in S^{1-j}_{1,0}(\re^n)$. Then we learn, similarly as above, 
\begin{equation}\label{eq0503a}
\bigabs{r_j(t,\cdot,\cdot)}_{-j,L,K}\leq C_{jLK}\jap{t}^{-\m}, \quad
\bigabs{\g_j(t,\cdot,\cdot)}_{1-j,L,K}\leq C_{jLK}, 
\end{equation}
uniformly in $t\in\re$ with any $L>0$, $K\Subset\re^n$, for each $j$.  
By \eqref{eq0503a} and the transport equations, we learn
\[
\g_{j,\pm}(x,\x)=\lim_{t\to\pm\infty} \g_j(t,x,\x)\in S^{1-j}_{1,0}(\re^n), 
\quad j=0,1,2,\dots,
\]
exist, and they converges with respect to the seminorms in $S^{1-j}_{1,0}(\re^n)$
by virtue of Lemma~\ref{lem0a01a}. More precisely, we have
\begin{equation}\label{eq0504a}
\bigabs{\g_j(t,\cdot,\cdot)-\g_{j,\pm}(\cdot,\cdot)}_{1-j,L,K}\leq C_{jLK}\jap{t}^{1-\m}, 
\quad t\in\re
\end{equation}
for all $j$. We note 
\[
\g_{0,\pm}(x,\x)=(a\circ w_\pm^{-1})(x,\x)
\]
by our construction. 

We now construct the asymptotic sum: $\g\sim\sum_{j=0}^\infty \g_j$ as follows. 
We choose $K\Subset\re^n$ so large that all symbols in the above construction 
are supported in $K\times\re^n$ for all $t$. We choose $\e_j>0$ so that 
\[
\sup\bigset{\jap{\x}^{j-2}\abs{\pa_x^\a\pa_\x^\b\g_j(t,x,\x)}}{|\a|+|\b|\leq j, x\in K, |\x|\geq \e_j^{-1}, 
t\in\re}\leq 2^{-j},
\]
which is possible since $|\g_j(t,\cdot,\cdot)|_{1-j,L,K}$ is uniformly bounded in $t$. 
We let $\chi\in C^\infty(\re^n)$ such that $\chi(\x)=0$ for $|\x|\leq 1$, and 
$\chi(\x)=1$ for $|\x|\geq 2$. Then we set
\[
\g(t,x,\x)=\sum_{j=0}^\infty \chi(\e_j\x)\g_j(t,x,\x).
\]
By the standard argument, we learn $\g(t,\cdot,\cdot)\in S^1_{1,0}(\re^n)$ for all $t\in\re$. 
Moreover, by the same argument with \eqref{eq0504a}, we learn 
\[
\g_\pm(x,\x)=\lim_{t\to\pm\infty}\g(t,x,\x)\in S^1_{1,0}(\re^n)
\]
exist, and they converge with respect to seminorms of $S^1_{1,0}(\re^n)$. 

\begin{lem}\label{lem0501a}
Let $\g$ as above, and let $r\in S^1_{1,0}(\re^n)$ such that 
\[
r^W\!(t,x,D_x)=\frac{\pa}{\pa t}\g^W\!(t,x,D_x) +i[L(t),\g^W\!(t,x,D_x)].
\]
Then  $r(t,\cdot,\cdot)\in S^{-\infty}_{1,0}(\re^n)$ for each $t\in\re$, and for any $N$, 
\[
\bigabs{r(t,\cdot,\cdot)}_{-N,L,K}\leq C_{NLK}\jap{t}^{-\m}.
\]
\end{lem}

\begin{proof}
We write $\tilde\g_j(t,x,\x)=\chi(\e_j\x)\g_j(t,x,\x)$ and set $\tilde r_j(t,x,\x)\in S^{-j}_{1,0}$ so that
\[
\frac{\pa}{\pa t}\tilde \g_j^W\!(t,x,D_x)+i[L(t),\tilde\g_j^W\!(t,x,D_x)]=\tilde r_j^W\!(t,x,D_x).
\]
By estimating $\pa_t\tilde\g_j$ and $[L(t),\tilde\g_j^W\!(t,x,D_x)]$ separately, we obtain 
rather crude estimates:
\begin{equation}\label{eq0505a}
\bigabs{\tilde r_j(t,\cdot,\cdot)}_{2-j,L,K}\leq C_{LK}\jap{t}^{-\m} 2^{-j}, \quad t\in\re,
\end{equation}
if $j\geq L+1$, where $C_{LK}$ is independent of $t$ and $j$. 
Similarly we have 
\begin{equation}\label{eq0506a}
\bigabs{\tilde r_j(t,\cdot,\cdot)-r_j(t,\cdot,\cdot)}_{-N,L,K}\leq C_{jNLK}\jap{t}^{-\m}
\end{equation}
for each $j$. Now we compute 
\begin{align*}
r^W\!(t,x,D_x) &=\sum_{j=0}^\infty \biggpare{\frac{\pa}{\pa t}\tilde\g_j^W\!(t,x,D_x)
+i\bigbrac{L(t),\tilde\g_j^W\!(t,x,D_x)}}\\
&=\1+\2+\3,
\end{align*}
where 
\begin{align*}
\1&=\sum_{j=0}^M \biggpare{\frac{\pa}{\pa t}\g_j^W\!(t,x,D_x)
+i\bigbrac{L(t),\g_j^W\!(t,x,D_x)}} \\
&=r_M^W\!(t,x,D_x),\\
\2&= \sum_{j=0}^M  \biggpare{\frac{\pa}{\pa t}(\tilde\g_j-\g_j)^W\!(t,x,D_x)
+i\bigbrac{L(t),(\tilde\g_j-\g_j)^W\!(t,x,D_x)}} \\
&=\sum_{j=0}^M \bigpare{(\tilde r_j-r_j)^W\!(t,x,D_x)}, \\
\3&= \sum_{j=M+1}^\infty  \biggpare{\frac{\pa}{\pa t}\tilde\g_j^W\!(t,x,D_x)
+i\bigbrac{L(t),\tilde\g_j^W\!(t,x,D_x)}} \\
&=\sum_{j=M+1}^\infty \tilde r_j^W\!(t,x,D_x),
\end{align*}
where $M= \max(N+2,L)$. 
Here we denote the symbol of an operator $A$ by $\s(A)$. Then we have
$|\s(\1)|_{-N,L,K}\leq C\jap{t}^{-\m}$
by \eqref{eq0503a}. Using \eqref{eq0506a}, we also have $
|\s(\2)|_{-N,L,K}\leq C\jap{t}^{-\m}$.
Finally we have 
\[
|\s(\3)|_{-N,L,K}\leq C\sum_{j=M+1}^\infty \jap{t}^{-\m} 2^{-j}
\leq C'\jap{t}^{-\m}
\]
by \eqref{eq0505a}, and the claim follows from these inequalities. 
\end{proof}

\begin{proof}[Proof of Theorem 1.3]
Let $R(t)=r^W\!(t,x,D_x)$. Lemma~\ref{lem0501a} implies 
\[
\norm{R(t)}_{\mathcal{L}(H^{-N},H^N)}\leq C_N\jap{t}^{-\m}, 
\quad t\in\re
\]
for any $N$. By our construction, we have 
\[
W(t)^{-1}G(t)W(t)-a^W\!(x,D_x)=\int_0^t W(s)^{-1}R(s) W(s)ds,
\]
where $G(t)=\g^W\!(t,x,D_x)$. Hence we have 
\begin{equation}\label{eq0507a}
W(t)^{-1}G(t) -a^W\!(x,D_x)W(t)^{-1} =\int_0^t W(s)^{-1}R(s)W(s)W(t)^{-1}ds.
\end{equation}
We note 
\begin{align*}
\bignorm{W(t)}_{\mathcal{L}(H^N,H^N)}
&=\bignorm{\jap{D_x}^N e^{itH_0}e^{-itH}\jap{D_x}^{-N}}_{\mathcal{L}(L^2)}\\
&\leq \bignorm{\jap{D_x}^N\jap{H}^{-N/2}}_{\mathcal{L}(L^2)}
\bignorm{\jap{H}^{N/2}\jap{D_x}^{-N}}_{\mathcal{L}(L^2)}
\end{align*}
is bounded uniformly in $t\in\re$ with any $N\in\re$. Hence we learn 
\[
\bignorm{W(s)^{-1}R(t)W(s) W(t)^{-1}}_{\mathcal{L}(H^{-N},H^N)}\leq C_N\jap{t}^{-\m}
\]
and then the RHS of \eqref{eq0507a} converges absolutely in $\mathcal{L}(H^{-N},H^N)$ as 
$t\to\pm \infty$. On the other hand, $G(t)$ converges to $\g^W_\pm\!(x,D_x)$ in 
$OPS^1_{1,0}(\re^n)$, and hence in $\mathcal{L}(H^1,L^2)$ as $t\to\pm\infty$. 
By the definition of wave operators, we then have 
\begin{align*}
&W(t)^{-1}G(t) \to W_\pm \g_\pm^W\!(x,D_x), \\
&a^W\!(x,D_x)W(t)^{-1}\to a^W\!(x,D_x)W_\pm
\end{align*}
strongly in $\mathcal{L}(H^1,H^{-1})$ as $t\to\pm\infty$. 
Thus we learn
\[
W_\pm \g_\pm^W\!(x,D_x) -a^W\!(x,D_x)W_\pm =R\in \mathcal{L}(H^{-N},H^N)
\]
with any $N$. Since $\g_\pm-a\circ w_\pm^{-1}\in S^0_{1,0}(\re^n)$, 
i.e., $a-\g_\pm\circ w_\pm\in S^0_{1,0}(\re^n)$, Theorem~\ref{thm0103a} now follows from 
Corollary~\ref{cor0202a}.
\end{proof}

%%%%%%%%%%%%%%%%%%%%%%%%%%%%%%%%%%%%%%%%%
%%%%%%%%%%%%  Appendix   %%%%%%%%%%%%%%%%
%%%%%%%%%%%%%%%%%%%%%%%%%%%%%%%%%%%%%%%%%

\appendix

\section{Classical trajectories}

Here we prove several technical inequalities. 

\begin{lem}\label{lem0a01a}
Suppose Assumption A with $\m>0$, and assume the global nontrapping condition. 
Let
\begin{align*}
&(z(t,x,\x),\y(t,x,\x))=w_0(t)(x,\x)=\exp(-tH_{p_0})\circ \exp(tH_k)(x,\x), \\
&(z_\pm(x,\x),\x_\pm(x,\x))=w_\pm(x,\x)=\lim_{t\to\pm\infty} w_0(t)(x,\x).
\end{align*}
Then for any $\a,\b\in\ze_+^n$ and $K\Subset \re^n$ there is $C_{\a\b K}>0$ such that  
\begin{equation}\label{eq0a01a}
\bigabs{\pa_x^\a\pa_\x^\b z(t,x,\x)}\leq C_{\a\b K}\jap{\x}^{-|\b|},\quad 
\bigabs{\pa_x^\a\pa_\x^\b \y(t,x,\x)}\leq C_{\a\b K}\jap{\x}^{1-|\b|}
\end{equation}
and, moreover, 
\begin{align}
&\bigabs{\pa_x^\a\pa_\x^\b (z(t,x,\x)-z_\pm(x,\x))}\leq C_{\a\b K}\jap{\x}^{-|\b|}
\jap{t\x}^{-\m+1}, \label{eq0a02a}\\
&\bigabs{\pa_x^\a\pa_\x^\b (\y(t,x,\x)-\x_\pm(x,\x))}\leq C_{\a\b K}\jap{\x}^{1-|\b|}
\jap{t\x}^{-\m}\label{eq0a03a}
\end{align}
for $x\in K$, $\x\in\re^n$. 
\end{lem}

\begin{proof}
We note 
\begin{equation}\label{eq0a04a}
\frac{\pa}{\pa t} z_j =\frac{\pa\ell_0}{\pa \y_j}(z,\y), \quad 
\frac{\pa}{\pa t}\y_j=-\frac{\pa \ell_0}{\pa z_j}(z,\y)
\end{equation}
and $\ell_0$ satisfies 
\[
\bigabs{\pa_z^\a\pa_\y^\b\ell_0(z,\y)}\leq 
C_{\a\b K}\jap{t\y}^{-\m-|\a|}\jap{\y}^{2-|\b|}
\]
for $z\in K\Subset \re^n$. We first show \eqref{eq0a01a} by induction in $|\a|+|\b|=m$. 
If $\a=\b=0$, \eqref{eq0a01a} is well-known (see, e.g, \cite{Na1}, Lemma~3.). 
Suppose \eqref{eq0a01a} holds for $|\a|+|\b|<m$, and let $|\a|+|\b|=m$. 
By differentiating \eqref{eq0a04a}, we have 
\begin{align}
&\frac{\pa}{\pa t}\bigpare{\pa_x^\a\pa_\x^\b z_j} 
=\sum_{k=1}^n \bigpare{\pa_x^\a\pa_\x^\b z_k}\frac{\pa^2\ell_0}{\pa z_k\pa\y_j} 
+\sum_{k=1}^n \bigpare{\pa_x^\a\pa_\x^\b \y_k}\frac{\pa^2\ell_0}{\pa\y_k\pa\y_j}+r_1,\label{eq0a05a}\\
&\frac{\pa}{\pa t}\bigpare{\pa_x^\a\pa_\x^\b \y_j} 
=-\sum_{k=1}^n \bigpare{\pa_x^\a\pa_\x^\b z_k}\frac{\pa^2\ell_0}{\pa z_k\pa z_j} 
-\sum_{k=1}^n \bigpare{\pa_x^\a\pa_\x^\b \y_k}\frac{\pa^2\ell_0}{\pa\y_k \pa z_j}+r_2,\label{eq0a06a}
\end{align}
where 
\[
r_1=O\bigpare{\jap{t\y}^{-\m}\jap{\y}^{1-|\b|}},\quad
r_2=O\bigpare{\jap{t\y}^{-1-\m}\jap{\y}^{2-|\b|}}
\]
by the induction hypothesis. We also note
\[
\frac{\pa^2\ell_0}{\pa z\pa\y}=O\bigpare{\jap{t\y}^{-1-\m}\jap{\y}}, \quad 
\frac{\pa^2\ell_0}{\pa\y^2}=O\bigpare{\jap{t\y}^{-\m}}, \quad
\frac{\pa^2\ell_0}{\pa z^2}=O\bigpare{\jap{t\y}^{-2-\m}\jap{\y}^2}.
\]
We consider the case: $t>0$. The case: $t<0$ is handled similarly. 
We now let $R\gg 0$, and $R\leq |\x|\leq 2R$. This also implies $R/C\leq |\y|\leq CR$ 
with some $C>0$. Then we have 
\begin{multline*}
\biggabs{\frac{\pa}{\pa t} \bigpare{\bigabs{\pa_x^\a\pa_\x^\b z}+R^{-1}\bigabs{\pa_x^\a\pa_\x^\b \y}}}
\leq C\jap{tR}^{-\m}R\bigpare{\bigabs{\pa_x^\a\pa_\x^\b z}+R^{-1}\bigabs{\pa_x^\a\pa_\x^\b\y}} \\
+C\jap{tR}^{-\m}R^{1-|\b|}. 
\end{multline*}
By the Duhamel formula and the estimate on the initial condition: 
\[
\bigpare{\bigabs{\pa_x^\a\pa_\x^\b z}+R^{-1}\bigabs{\pa_x^\a\pa_\x^\b \y}}
\Big|_{t=0}\leq C R^{-|\b|},
\]
we learn 
\begin{align*}
\bigabs{\pa_x^\a\pa_\x^\b z}+R^{-1}\bigabs{\pa_x^\a\pa_\x^\b \y}&\leq 
C \exp\biggpare{C\int_0^\infty\jap{tR}^{-\m}Rdt}\times \\
&\qquad \times \biggpare{C R^{-|\b|}+C\int_0^\infty\jap{tR}^{-\m}R^{1-|\b|}dt}\\
&\leq C' R^{-|\b|}
\end{align*}
since $\m>1$. This implies \eqref{eq0a01a}. 

Now we use \eqref{eq0a05a}--\eqref{eq0a06a} again with \eqref{eq0a01a} to learn 
\begin{align*}
\biggabs{\frac{\pa}{\pa t} \bigpare{\bigabs{\pa_x^\a\pa_\x^\b z}}} &\leq C\jap{tR}^{-1-\m}R^{1-|\b|}
+C \jap{tR}^{-\m}R^{1-|\b|}+C\jap{tR}^{-\m}R^{1-|\b|}\\
&\leq C' \jap{tR}^{-\m}R^{1-|\b|}, \\
\biggabs{\frac{\pa}{\pa t} \bigpare{\bigabs{\pa_x^\a\pa_\x^\b \y}}}&\leq 
C\jap{tR}^{-2-\m}R^{2-|\b|}+C\jap{tR}^{-1-\m}R^{2-|\b|}+ C\jap{tR}^{-1-\m}R^{2-|\b|}\\
&\leq C'\jap{tR}^{-1-\m}R^{2-|\b|}.
\end{align*}
Hence, by integrating these on $[t,\infty)$, we learn 
\begin{align*}
\bigabs{\pa_x^\a\pa_\x^\b(z(t,x,\x)-z_+(x,\x))} &\leq C\int_t^\infty\jap{tR}^{-\m}R^{1-|\b|}dt \\
&=CR^{-|\b|}\int_{Rt}^\infty \jap{s}^{-\m}ds 
\leq CR^{-|\b|} \jap{tR}^{1-\m}, \\
\bigabs{\pa_x^\a\pa_\x^\b(\y(t,x,\x)-\x_+(x,\x))} &\leq C\int_t^\infty\jap{tR}^{-1-\m}R^{2-|\b|}dt \\
&=CR^{1-|\b|}\int_{Rt}^\infty \jap{s}^{-1-\m}ds 
\leq CR^{1-|\b|} \jap{tR}^{-\m}, 
\end{align*}
For the case: $t<0$, we integrate these inequalities on $(-\infty,t]$ to obtain corresponding estimates. 
\eqref{eq0a02a} and \eqref{eq0a03a} follows immediately from these estimates. 
\end{proof}

Next we consider the evolution:
\[
(z_1(t,x,\x),\y_1,t,x,\x))=w(t)(x,\x)=\exp(-tH_{p_0})\circ \exp(tH_p)(x,\x)
\]
and compare it with $w_0(t)(x,\x)$ as $|\x|\to\infty$ for $t$ in a fixed bounded interval. 
We denote $I_T=[-T,T]$, and we consider the case $t>0$ only in the proof.  

\begin{lem}\label{lem0a02a}
Suppose Assumption~\ref{ass0aa} with $1<\m<2$, and assume the global nontrapping condition. 
Let $T>0$. Then for any $\a,\b\in\ze_+^n$ and $K\Subset \re^n$, there is 
$C=C(\a,\b,K,T)>0$ such that 
\begin{equation}\label{eq0a07a}
\bigabs{\pa_x^\a\pa_\x^\b z_1(t,x,\x)}\leq C\jap{\x}^{-|\b|}, \quad 
\bigabs{\pa_x^\a\pa_\x^\b \y_1(t,x,\x)}\leq C\jap{\x}^{1-|\b|}
\end{equation}
and 
\begin{align}
&\bigabs{\pa_x^\a\pa_\x^\b (z_1(t,x,\x)-z(t,x,\x))}\leq C\jap{\x}^{1-\m-|\b|}, \label{eq0a08a}\\
&\bigabs{\pa_x^\a\pa_\x^\b (\y_1(t,x,\x)-\y(t,x,\x))}\leq C\jap{\x}^{1-\m-|\b|}\label{eq0a09a}
\end{align}
for $t\in I_T$, $x\in K$ and $\x\in\re^n$. 
\end{lem}

\begin{proof}
{\bf Step 1: } We first show \eqref{eq0a08a} and \eqref{eq0a09a} with $\a=\b=0$.
We denote $\tilde V(t,z,\y)=V(z-t\y)$ so that $\ell(t,z,\y)=\ell_0(t,z,\y)+\tilde V(t,z,\y)$. 
We note 
\begin{align}
&\frac{\pa}{\pa t} (z_1-z) =\frac{\pa \ell_0}{\pa\y}(z_1,\y_1)-\frac{\pa\ell_0}{\pa \y}(z,\y)
+\frac{\pa \tilde V}{\pa \y}(z_1,\y_1), \label{eq0a10a}\\
&\frac{\pa}{\pa t} (\y_1-\y) =-\frac{\pa \ell_0}{\pa z}(z_1,\y_1)+\frac{\pa\ell_0}{\pa z}(z,\y)
-\frac{\pa \tilde V}{\pa z}(z_1,\y_1).\label{eq0a11a}
\end{align}
We now suppose 
\begin{equation}\label{eq0a12a}
|z_1-z|\leq \e_0|z|, \quad 
|\y_1-\y|\leq \e_0|\y|
\end{equation}
for $x\in K\Subset \re^n$, $\x\in\re^n$ and $t\in I_{T'}$ with $T'>0$. We suppose $t>0$
and we have 
\begin{align}\label{eq0a13a}
\frac{\pa}{\pa t} (z_1-z) &= (z_1-z)\cdot 
\int_0^1 \frac{\pa^2 \ell_0}{\pa z\pa\y}(sz_1+(1-s)z,s\y_1+(1-s)\y)ds\\
&\quad + (\y_1-\y)\cdot \int_0^1\frac{\pa^2\ell_0}{\pa\y\pa\y}
(sz_1+(1-s)z,s\y_1+(1-s)\y)ds\nonumber \\
&\quad +\frac{\pa \tilde V}{\pa\y}(z_1,\y_1), \nonumber\\
\label{eq0a14a}\frac{\pa}{\pa t} (\y_1-\y) &= -(z_1-z)\cdot 
\int_0^1 \frac{\pa^2 \ell_0}{\pa z\pa z}(sz_1+(1-s)z,s\y_1+(1-s)\y)ds\\
&\quad - (\y_1-\y)\cdot \int_0^1\frac{\pa^2\ell_0}{\pa\y\pa z}
(sz_1+(1-s)z,s\y_1+(1-s)\y)ds\nonumber \\
&\quad -\frac{\pa \tilde V}{\pa z}(z_1,\y_1), \nonumber
\end{align}
for $t\in[0,T']$. We again assume $R\leq |\x|\leq 2R$ with $R\gg0$, so that 
$|\y|\sim O(|\x|)=O(R)$. Then we have 
\begin{align*}
\frac{\pa}{\pa t}(|z_1-z|+R^{-1}|\y_1-\y|) &\leq C\jap{tR}^{-\m}R(|z_1-z|+R^{-1}
|\y_1-\y|) \\
&\quad + C\jap{tR}^{2-\m}R^{-1}.
\end{align*}
Then by the Duhamel formula, we learn 
\begin{align*}
|z_1-z|+R^{-1}|\y_1-\y| &\leq C e^{C\int_0^t \jap{sR}^{-\m}Rds}\times 
\int_0^t \jap{sR}^{2-\m}R^{-1}ds\\
&\leq C'R^{-2}\int_0^{Rt}\jap{s}^{2-\m}ds,
\end{align*}
since $z_1=z$, $\y_1=\y$ at $t=0$. We note 
\[
\int_0^\s \jap{s}^{2-\m}ds \leq \int_0^\s(1+s)^{2-\m}ds =\frac{(1+\s)^{3-\m}-1}{3-\m}
\leq C(\s+\s^{3-\m}), 
\]
and hence 
\begin{equation}\label{eq0a15a}
|z_1-z|+R^{-1}|\y_1-\y|\leq CR^{-2}(Rt+(Rt)^{3-\m})\leq C' t R^{1-\m}
\end{equation}
for $t\in [0,T']$, $R\gg 0$. Thus, in particular,  \eqref{eq0a12a} holds with 
$\e_0=O(R^{1-\m})$. By contradiction, we learn that \eqref{eq0a12a} holds for $t\in [0,T]$ if 
$|\x|$ is sufficiently large. \eqref{eq0a15a} also implies \eqref{eq0a08a} with $\a=\b=0$. 
We substitute \eqref{eq0a15a} to \eqref{eq0a14a} to learn 
\begin{align*}
\biggabs{\frac{\pa}{\pa t}(\y_1-\y)} &\leq C\biggpare{t\jap{tR}^{-2-\m}R^{3-\m}+
t\jap{tR}^{-1-\m}R^{3-\m}+\jap{tR}^{1-\m}}\\
&\leq C\biggpare{\jap{tR}^{-1-\m}R^{2-\m}+ \jap{tR}^{-\m}R^{2-\m}+\jap{tR}^{1-\m}}
\end{align*}
since $t\jap{tR}^{-1}\leq R^{-1}$. Integrating this inequality, we have 
\begin{align*}
|\y_1-\y| &\leq C\biggpare{R^{1-\m}\int_0^\infty \jap{tR}^{-\m}  Rdt+\int_0^t \jap{sR}^{1-\m}ds} \\
&\leq C'\biggpare{R^{1-\m}+\frac{\jap{tR}^{2-\m}}{R}}\leq C'' R^{1-\m}
\end{align*}
for $t\in[0,T]$, and this implies \eqref{eq0a09a} with $\a=\b=0$. 

{\bf Step 2: } We then prove \eqref{eq0a07a} mimicking the proof of \eqref{eq0a01a}. We note 
\[
\bigabs{\pa_z^\a\pa_\y^\b\ell(z,\y)}
\leq C\bigpare{\jap{t\y}^{-\m-|\a|}\jap{\y}^{2-|\b|}+\jap{t\y}^{2-\m-|\a|}\jap{\y}^{-|\b|}}
\]
for any $\a,\b\in\ze_+^n$. We prove them by induction in $|\a|+|\b|=m$. 
We suppose \eqref{eq0a07a} holds for $|\a|+|\b|<m$ and let $|\a|+|\b|=m$. 
Analogously to the proof of \eqref{eq0a01a}, it follows from the induction step that 
if $R\leq|\x|\le 2R$ with $R\gg0$, 
\begin{align*}
&\frac{\pa}{\pa t} \bigpare{\bigabs{\pa_x^\a\pa_\x^\b z_1}
+R^{-1}\bigabs{\pa_x^\a\pa_\x^\b\y_1}}\\
&\quad \leq C\bigpare{\jap{tR}^{-\m}R+\jap{tR}^{2-\m}R^{-1}}
\bigpare{\bigabs{\pa_x^\a\pa_\x^\b z_1}
+R^{-1}\bigabs{\pa_x^\a\pa_\x^\b\y_1}}\\
&\quad\quad  +C\bigpare{\jap{tR}^{-\m}R^{1-|\b|}+\jap{tR}^{2-\m}R^{-1-|\b|}}.
\end{align*}
Hence, by using the Duhamel formula again, we obtain 
\begin{align*}
&\bigabs{\pa_x^\a\pa_\x^\b z_1}+R^{-1}\bigabs{\pa_x^\a\pa_\x^\b\y_1}\\
&\quad\leq C\exp\biggpare{C\int_0^t \jap{sR}^{-\m}Rds +C\int_0^t\jap{sR}^{2-\m}R^{-1}ds}\times \\
&\quad\quad \times \biggpare{R^{-|\b|}+\int_0^t \jap{sR}^{-\m}R^{1-|\b|}ds 
+\int_0^t \jap{sR}^{2-\m}R^{-1-|\b|}ds }\\
&\quad \leq C' R^{-|\b|}
\end{align*}
for $t\in[0,T]$, and \eqref{eq0a07a} follows. 

{\bf Step 3: } We now prove \eqref{eq0a10a} and \eqref{eq0a11a}. We again prove it by induction in $m=|\a|+|\b|$. 
We suppose 
\begin{equation}\label{eq0a16a}
\bigabs{\pa_x^\a\pa_\x^\b (z_1-z)}\leq C|t|\jap{\x}^{1-\m-|\b|}, \quad 
\bigabs{\pa_x^\a\pa_\x^\b (\y_1-\y)}\leq C|t|\jap{\x}^{2-\m-|\b|}
\end{equation}
hold if $|\a|+|\b|<m$. We differentiate \eqref{eq0a13a} and \eqref{eq0a14a} to learn 
\begin{align*}
&\frac{\pa}{\pa t}(\pa_x^\a\pa_\x^\b(z_1-z)) \\
&\quad= \pa_x^\a\pa_\x^\b (z_1-z)\cdot 
\int_0^1 \frac{\pa^2\ell_0}{\pa z\pa\y}(sz_1+(1-s)z,s\y_1+(1-s)\y)ds \\
&\qquad+ \pa_x^\a\pa_\x^\b (\y_1-\y)\cdot 
\int_0^1 \frac{\pa^2\ell_0}{\pa\y\pa\y}(sz_1+(1-s)z,s\y_1+(1-s)\y)ds \\
&\qquad +r_1,
\end{align*}
where 
\begin{align*}
|r_1|&\leq C\bigpare{|t|\jap{tR}^{-1-\m}R^{2-\m-|\b|}+|t|\jap{tR}^{-\m} R^{2-\m-|\b|}
+\jap{tR}^{2-\m}R^{-1-|\b|}}\\
&\leq C'\bigpare{\jap{tR}^{1-\m}R^{1-\m-|\b|}+\jap{tR}^{2-\m}R^{-1-|\b|}}, 
\end{align*}
and 
\begin{align}\label{eq0a17a}
&\frac{\pa}{\pa t}(\pa_x^\a\pa_\x^\b(\y_1-\y)) \\
&\quad= -\pa_x^\a\pa_\x^\b (z_1-z)\cdot 
\int_0^1 \frac{\pa^2\ell_0}{\pa z\pa z}(sz_1+(1-s)z,s\y_1+(1-s)\y)ds \nonumber\\
&\qquad+ \pa_x^\a\pa_\x^\b (\y_1-\y)\cdot 
\int_0^1 \frac{\pa^2\ell_0}{\pa\y\pa z}(sz_1+(1-s)z,s\y_1+(1-s)\y)ds \nonumber \\
&\qquad +r_2,\nonumber
\end{align}
where 
\begin{align*}
|r_2|&\leq C\bigpare{|t|\jap{tR}^{-2-\m}R^{3-\m-|\b|}
+|t|\jap{tR}^{-1-\m}R^{3-\m-|\b|} +\jap{tR}^{1-\m}R^{-|\b|}}\\
&\leq C'\bigpare{\jap{tR}^{-\m}R^{2-\m-|\b|}+\jap{tR}^{1-\m}R^{-|\b|}},
\end{align*}
by virtue of \eqref{eq0a16a}. These imply
\begin{align*}
&\frac{\pa}{\pa t}\bigpare{\bigabs{\pa_x^\a\pa_\x^\b(z_1-z)}
+R^{-1}\bigabs{\pa_x^\a\pa_\x^\b(\y_1-\y)}}\\
&\quad \leq C\bigpare{\bigabs{\pa_x^\a\pa_\x^\b(z_1-z)}
+R^{-1}\bigabs{\pa_x^\a\pa_\x^\b(\y_1-\y)}}\jap{tR}^{-\m}R \\
&\qquad +C\bigpare{\jap{tR}^{1-\m}R^{1-\m-|\b|}+\jap{tR}^{2-\m}R^{-1-|\b|}},
\end{align*}
and by the Duhamel formula again with the vanishing initial conditions, we obtain 
\[
\bigabs{\pa_x^\a\pa_\x^\b(z_1-z)}+R^{-1}\bigabs{\pa_x^\a\pa_\x^\b(\y_1-\y)}
\leq C|t|R^{1-\m-|\b|}
\]
for $t\in[0,T]$. This proves \eqref{eq0a16a} for $|\a|+|\b|=m$, and we learn \eqref{eq0a16a} holds for 
all $\a,\b$. \eqref{eq0a08a} follows immediately from \eqref{eq0a16a}. We substitute \eqref{eq0a16a} to \eqref{eq0a17a}, 
and we have \eqref{eq0a09a} analogously to Step 1.
\end{proof}

%%%%%%%%%%%%%%%%%%%%%%%%%%%%%%%%%%%%%%%%%
%%%%%%%%%%%%  References  %%%%%%%%%%%%%%%%
%%%%%%%%%%%%%%%%%%%%%%%%%%%%%%%%%%%%%%%%%


\begin{thebibliography}{99}

\bibitem{CKS}  Craig, W., Kappeler, T., Strauss, W.: 
Microlocal dispersive smoothing for the Schr\"odinger equation,
\textit{Comm.\ Pure Appl.\ Math.}\ {\bf 48} (1995), 769--860.

\bibitem{Fu} Fujiwara, D.: Remarks on convergence of the Feynman path integrals, 
{\it Duke Math.\ J.} {\bf 47} (1980), 559--600.

\bibitem{HW1} Hassell, A., Wunsch, J.: 
On the structure of the Schr\"odinger propagator, 
in {\it Partial Differential Equations and Inverse Problems, Contemp.\ Math.} {\bf 362}, 
199--209, A.M.S., Providence, RI, 2004.

\bibitem{HW2} Hassell, A., Wunsch, J.: 
The Schr\"odinger propagator for scattering metrics. 
{\it Ann.\ Math.}\ {\bf 162} (2005), 487--523.

\bibitem{Ho1} H\"ormander, L.: Fourier integral operators I, {\it Acta Math.}\ {\bf 127} (1971), 79--183. 

\bibitem{Ho2} H\"ormander, L.: 
{\it The Analysis of Linear Partial Differential Operators}\ I--IV, Springer-Verlag, New York, 1983--1985. 
\bibitem{I} Ito, K.: 
Propagation of singularities for Schr\"odinger equations
on the Euclidean space with a scattering metric,
\textit{Comm.\ Partial Differential Equations}
{\bf 31} (2006), 1735--1777.

\bibitem{IN1} Ito, K., Nakamura, S.: 
Singularities of solutions to Schr\"odinger equation on scattering manifold. Preprint, Nov. 2007. 
To appear in {\it American J. Math.}

\bibitem{KS} Kapitanski L., Safarov, Y.: 
A parametrix for the nonstationary Schr\"odinger equation,
in {\it Differential Operators and Spectral Theory, Amer.\ Math.\ Soc.\ Transl.}\ {\bf 189},
139--148, A. M. S., Providence, RI, 1999.

\bibitem{MNS1} Martinez, A., Nakamura, S., Sordoni, V.: 
Analytic smoothing effect for the Schr\"odinger equation with long-range perturbation, 
{\it Comm.\ Pure Appl.\ Math.} {\bf 59} (2006), 1330--1351.

\bibitem{MNS2} Martinez, A., Nakamura, S., Sordoni, V.: 
Analytic wave front for solutions to Schr\"odinger equation, 
{\it Advances in Math.} {\bf 222} (2009), 1277--1307.

\bibitem{Me} Melrose, R.: Spectral and scattering theory for the Laplacian on asymptotically Euclidian
spaces, in {\it Spectral and Scattering Theory} (Sanda, 1992), Lecture Notes in Pure and
Appl. Math. {\bf 161}, 85--130, Dekker, New York, 1994.

\bibitem{Na0} Nakamura, S.: 
Propagation of the homogeneous wave front Set for Schr\"odinger equations,
\textit{Duke Math.\ J.} {\bf 126} (2005), 349--367.

\bibitem{Na1} Nakamura, S.: Wave front set for solutions to Schr\"odinger equations. 
{\it J. Functional Analysis} {\bf 256} (2009), 1299-1309. 

\bibitem{Na2} Nakamura, S.: 
Semiclassical singularity propagation property for Schr\"odinger equations. 
{\it J. Math.\ Soc.\ Japan}\  {\bf 61} (2009), 177-211.

\bibitem{RZ} Robbiano, L., Zuily, C.: 
Microlocal analytic smoothing effect for Schr\"odinger  equation,
\textit{Duke Math.\ J.} {\bf 100} (1999), 93--129.

\bibitem{So} Sogge, C.: 
{\it Fourier Integrals in Classical Analysis}, Cambridge Univ.\ Press 1993. 

\bibitem{W} Wunsch, J.: 
Propagation of singularities and growth for Schr\"odinger operators,
\textit{Duke Math.\ J.} {\bf 98} (1999), 137--186.

\bibitem{Y} Yajima, K.:
Smoothness and nonsmoothness of the fundamental solution of time dependent Schr\"odinger equations, 
{\it Comm.\ Math.\ Phys.} {\bf 181} (1996), 605--629.

\end{thebibliography}
\end{document}